\documentclass[11pt]{amsart}
\usepackage{amscd,amsfonts,amsmath,amssymb,amstext,amsthm, xcolor, mathrsfs} 

\date{}
\def\dim{\operatorname{dim}}

\input xy
\xyoption{all}
\theoremstyle{plain}
\newtheorem{theorem} {Theorem} [section]
\newtheorem{lemma} [theorem]{Lemma}
\newtheorem{proposition}[theorem]{Proposition}
\newtheorem{corollary} [theorem]{Corollary}
\theoremstyle{definition}
\newtheorem{definition} [theorem] {Definition}
\newtheorem{example}[theorem] {Example}

  \title[Local holomorphic mappings respecting homogeneous subspaces]{Local holomorphic mappings respecting homogeneous subspaces on rational homogeneous spaces}

 \author[J. Hong]{Jaehyun Hong}
 \address{Center for Mathematical Challenges, Korea Institute for  Advanced  Study, 85 Hoegiro, Dongdaemun-gu, Seoul 02455, Korea}
  \email{jhhong00@kias.re.kr}

  \author[S.-C. Ng]{Sui-Chung Ng}
 \address{School of Mathematical Sciences, Shanghai Key Laboratory of PMMP, East China Normal University, Shanghai, China}
  \email{scng@math.ecnu.edu.cn}

\begin{document}

\begin{abstract}
Let $G/P$ be a rational homogeneous space (not necessarily irreducible) and $x_0\in G/P$ be the point at which the isotropy group is $P$. The $G$-translates of the orbit $Qx_0$ of a parabolic subgroup $Q\subsetneq G$ such that $P\cap Q$ is parabolic are called \textit{$Q$-cycles}. We established an extension theorem for local biholomorphisms on $G/P$ that map local pieces of $Q$-cycles into $Q$-cycles. We showed that such maps extend to global biholomorphisms of $G/P$ if $G/P$ is $Q$-cycle-connected, or equivalently, if there does not exist a non-trivial parabolic subgroup containing $P$ and $Q$. Then we applied this to the study of local biholomorphisms preserving the real group orbits on $G/P$ and showed that such a map extend to a global biholomorphism if the real group orbit admits a \textit{non-trivial holomorphic cover} by the $Q$-cycles. The non-closed boundary orbits of a bounded symmetric domain embedded in its compact dual are examples of such real group orbits. Finally, using the results of Mok-Zhang on Schubert rigidity, we also established a Cartan-Fubini type extension theorem pertaining to $Q$-cycles, saying that if a local biholomorphism preserves the variety of tangent spaces of $Q$-cycles, then it extends to a global biholomorphism when the $Q$-cycles are positive dimensional and $G/P$ is of Picard number 1. This generalizes a well-known theorem of Hwang-Mok on minimal rational curves.
\end{abstract}

  \maketitle

\section{Introduction}

\subsection{Extension theorem for $Q$-cycles-respecting maps} $\,$

Starting from the late 90s, Hwang and Mok has developed a theory to study Fano manifolds of Picard number 1 using their rational curves. A number of difficult problems in Algebraic Geometry have been solved using the theory, for example, those related to deformation rigidity of rational homogeneous spaces~\cite{HM98, HM05}, Lazarsfeld's problem~\cite{HM99} and target rigidity~\cite{FuHwang}.
The theory is based on a special kind of rational curves, called   \textit{minimal rational curves}, in which the word ``minimal" means the degree of the rational curve with respect to some choice of ample line bundle is minimal among  free rational curves.
One of the basic themes of the theory is the rigidity phenomena of the holomorphic mappings that interact nicely with minimal rational curves. To make things more concrete, it suffices for us to mention the so-called \textit{Cartan-Fubini type extension}, which is a key ingredient in Hwang-Mok's theory and its application to geometric problems. In the case of rational homogeneous spaces of Picard number 1 other than $\mathbb P^n$, an equidimensional version of the extension says that if a local biholomorphism preserves the set of tangent directions of minimal rational curves, then it extends to a global biholomorphism. There are generalizations to general uniruled projective manifolds~\cite{HM01} and also non-equidimensional situations~\cite{HongMok10, Hwang15}.

Roughly speaking, there are two components in the proofs for Cartan-Fubini type extensions mentioned above. Firstly, one proves that if a local holomorphic map preserves the set of tangent directions of minimal rational curves, then it actually maps pieces of minimal rational curves into minimal rational curves. We will say that such maps \textit{respect minimal rational curves}. Then one proves an extension statement for the maps that respect minimal rational curves.
In this article, we are going to study such extension phenomena for local biholomorphisms that respect homogeneous submanifolds (or cycles) of arbitrary dimension, in the case of  rational homogeneous spaces of arbitrary Picard number. For maps that are only known to respect tangencies, our results together with those from the study of Schubert rigidity  will also lead to rigidity statements for spaces of Picard number 1, which are analogous to the usual Cartan-Fubini type extensions. This will be addressed at the end of this section.
Before stating our main results, we would like to point out that this kind of cycle-respecting properties appear naturally and ubiquitously in some   areas in Several Complex Variables, e.g. in the study of proper holomorphic mappings among symmetric domains and holomorphic mappings that preserves real group orbits on rational homogeneous spaces. For example, since the work of Mok-Tsai~\cite{MokTsai} it is now well known that any proper holomorphic map between bounded symmetric domains respects symmetric subspaces of a particular kind; and in~\cite{NG12} it was discovered that in many cases a holomorphic mappings preserving the closed orbits of $SU(p,q)$ on Grassmannians respect certain subgrassmannians.

From now on, we let $G/P$ be a rational homogeneous space, where $G$ is a semisimple complex Lie group and $P\subset G$ is a parabolic subgroup. Let $x\in G/P$ be the base point at which the isotropy group is $P$ and let $Q\subset G$ be another parabolic subgroup such that $P\cap Q$ is parabolic. Then the orbit $Qx$ of $Q$ (under canonical left action) is a rational homogeneous subspace of $G/P$. This orbit and any of its translations by   elements of $G$ will be called a \textit{$Q$-cycle}. A $Q$-cycle is non-trivial (not a point nor equal to $G/P$) if $Q\not\subset P$ and $Q\neq G$.
We remark that when $G/P$ is of Picard number 1 and is associated to a long root, then the minimal rational curves on $G/P$  are the $Q$-cycles for some $Q\subset G$. In the Cartan-Fubini type extensions established in the literature for minimal rational curves, the condition on Picard number (equal to 1) is always assumed for the source manifold, because otherwise two general points in the source manifold cannot be connected by a chain of minimal rational curves, which is certainly an obstruction for extension. In our situation regarding the $Q$-cycles in $G/P$, we say that $G/P$ is \textit{$Q$-cycle-connected} if any two points can be connected by a chain of $Q$-cycles and we have the following simple criterion for $Q$-cycle-connectivity which is very easy to check on the Dynkin diagram:

\noindent\textbf{$Q$-cycle-connectivity.} (see Proposition~\ref{connectivity})
\textit{
$G/P$ is $Q$-cycle-connected if and only if there does not exist any parabolic subgroup $R\subsetneq G$ containing $P$ and $Q$, or equivalently, when the set of marked nodes on the Dynkin diagram associated to $P$ and $Q$ are disjoint.
}

A local biholomorphism defined on $G/P$ will be said to be \textit{$Q$-cycle-respecting} if it maps pieces of $Q$-cycles into $Q$-cycles (for the precise definition, see Section~\ref{sheaf}). Now we are able to state our main result regarding the extension of the germs of such maps:

\begin{theorem}\label{mainthm1}
Let $G/P$ be a rational homogeneous space and $Q\subsetneq G$ be a parabolic subgroup such that $P\cap Q$ is parabolic. If $G/P$ is $Q$-cycle-connected, or equivalently, if there does not exist any parabolic subgroup $R\subsetneq G$ containing $P$ and $Q$, then every germ of $Q$-cycle-respecting map on $G/P$ extends to a biholomorphism of $G/P$.
\end{theorem}

Comparing with the previous works on Cartan-Fubini type extension, we first of all do not need to impose the Picard number 1 condition and in particular, our  theorem applies also to reducible rational homogeneous spaces. As mentioned, the usual Picard number 1 condition was assumed in order to get cycle-connectivity by minimal rational curves. Secondly, in these studies, after establishing the algebraic extension of the concerned local map, the Picard number 1 condition and the fact that the cycles are one-dimensional had also been used in an essential way to prove the univalence and holomorphicity of the extension. Here in our theorem, the cycle-connectivity is taken as the only hypothesis, which has an equivalent group-theoretic condition and can be easily verified.

The novelty of our method is that instead of establishing the extension of a $Q$-cycle-respecting germ $\mathfrak f$ directly on $G/P$ (or on the graph of $\mathfrak f$), in which one usually only obtains \textit{meromorphic} maps, we use the germ $\mathfrak f$ to construct a sequence of \textit{holomorphic} maps $\{F_k\}$ from a sequence of projective manifolds $\{\mathcal T^k_{x_0}\}$, called the \textit{$Q$-towers}, to the target space. We want to emphasize that the maps $\{F_k\}$ are holomorphic \textit{by construction} and thus we do not need to deal with indeterminacies nor essential singularities of meromorphic maps. Geometrically, the $Q$-towers can be interpreted as some sort of universal families of the chains of $Q$-cycles emanating from $x_0$ and there are evaluation maps ${\bf p}_k:\mathcal T^k_{x_0}\rightarrow G/P$ sending the towers to their images in $G/P$ (see~\cite{MokZhang} for the case of rational curves). The $Q$-cycle-connectivity hypothesis then implies that for some positive integer $N$, the evaluation map ${\bf p}_N:\mathcal T^N_{x_0}\rightarrow G/P$ is surjective. The second part of our proof is then to show that $F_N$ actually descends to a holomorphic map $F:G/P\rightarrow G/P$, by exploiting the theory of $P$-action on $G/P$.

\subsection{Local biholomorphisms preserving real group orbits on $G/P$} $\,$\

We then apply Theorem~\ref{mainthm1} to the study of holomorphic maps pertaining to the real group orbits on $G/P$. We recall first of all that a real Lie subgroup $G_0\subset G$ is called a real form of $G$ if the complexification of its Lie algebra $\mathfrak g_0$ is the Lie algebra $\mathfrak g$ of $G$. Wolf has laid down the foundation of the action of $G_0$ on $G/P$ in~\cite{W69}. It is now well known that $G_0$ has only a finite number of orbits on $G/P$, which will be called \textit{real group orbits}, including in particular open orbit(s) and a unique closed orbit. Any open real group orbit is also called a \textit{flag domain} in the literature. The bounded symmetric domains embedded in their compact duals are the most well known examples of flag domains. In Several Complex Variables, the study of proper holomorphic maps on domains is a classical topic and one is naturally led to consider local holomorphic maps that respect the boundary structures of the domains. In the case of open real group orbits  (i.e. flag domains), their boundary is again a union of real group orbits which are also homogeneous CR submanifolds in $G/P$. Thus, local holomorphic maps preserving real group orbits (or CR maps between real group orbits) arise very naturally in Several Complex Variables.

Our key observation here is that, in many cases a real group orbit $\mathcal O\subset G/P$ can be covered by a family of $Q$-cycles in a very special way that $\mathcal O$ and these $Q$-cycles are ``tangled" under the holomorphic mappings from $\mathcal O$ or into $\mathcal O$. For the precise definition, see Definition~\ref{holomorphic cover}. In such cases, we will say that $\mathcal O$ has a \textit{holomorphic cover of $Q$-type} and the cover is said to be \textit{non-trivial} if the covering $Q$-cycles are neither points nor the entire $G/P$.  Here we simply remark that
the non-closed boundary orbits of a bounded symmetric domain embedded in its compact dual are examples of real group orbits having such kind of non-trivial covers. For more examples, see Example~\ref{example}. Our main result for the real group orbits on a rational homogeneous space $G/P$ can now be stated as follows.

\begin{theorem} \label{theorem extension of subdiagram type}
Let $\mathcal O$ be a real group orbit on $G/P$ and $U\subset G/P$ be a connected open set such that $U\cap\mathcal O\neq\emptyset$. Suppose $\mathcal O$ has a non-trivial holomorphic cover of $Q$-type for some parabolic subgroup $Q\subset G$ and $G/P$ is $Q$-cycle-connected. If $f:U\rightarrow f(U)\subset G/P$ is a biholomorphism such that $f(U\cap\mathcal O)\subset\mathcal O$, then $f$ extends to a biholomorphism of $G/P$.
\end{theorem}

As an illustration of Theorem~\ref{theorem extension of subdiagram type}, we specialize it for the case of bounded symmetric domains (which are allowed to be reducible):

\begin{corollary}
Let $M$ be a compact Hermitian symmetric space and $\Omega\subset M$ be the Borel embedding of the bounded symmetric domain $\Omega$ dual to $M$. Let $G_0:=Aut(\Omega)\subset Aut(M)$ and $\mathcal O\subset\partial\Omega$ be $G_0$-orbit which is neither open nor closed, and $U\subset M$ be a connected open set such that $U\cap\mathcal O\neq\emptyset$. If $f:U\rightarrow f(U)\subset M$ is a biholomorphism such that $f(U\cap\mathcal O)\subset\mathcal O$, then $f$ extends to a biholomorphism of $M$.
\end{corollary}

The closed $G_0$-orbit in $\partial\Omega$ is precisely the Shilov boundary of $\Omega$ and it is a classical result of Alexander~\cite{Alexander} and Khenkin-Tumanov~\cite{HT} that the same extension holds in this case for irreducible bounded symmetric domains of at least dimension 2. For certain $G_0$-orbits in a compact Hermitian symmetric space, similar extensions have also been established in the works including  \cite{KaZ00}, \cite{KaZ06}, \cite{KiZ15}, \cite{MN}.

In Theorem~\ref{theorem extension of subdiagram type}, if $G/P$ is of Picard number 1, then cycle connectivity is automatic whenever the $Q$-cycles are non-trivial, so we have

\begin{corollary}\label{picard number 1 corollary}
Let $\mathcal O$ be a real group orbit on $G/P$ of Picard number 1 and $U\subset G/P$ be a connected open set such that $U\cap\mathcal O\neq\emptyset$. Suppose $\mathcal O$ has a holomorphic cover of $Q$-type for some parabolic subgroup $Q\subsetneq G$ and $Q\not\subset P$. If $f:U\rightarrow f(U)\subset G/P$ is a biholomorphism such that $f(U\cap\mathcal O)\subset\mathcal O$, then $f$ extends to a biholomorphism of $G/P$.
\end{corollary}

Corollary~\ref{picard number 1 corollary} covers in particular the closed $SU(p,q)$-orbits on the Grassmannian $Gr(d,\mathbb C^{p+q})$, where $d<\min(p,q)$. These orbits include the boundaries of the so-called \textit{generalized balls}, which have been studied in~\cite{BH, NG12}, etc.

\subsection{Cartan-Fubini type extension for $Q$-cycles} $\,$

As mentioned at the beginning of the introduction, the original Cartan-Fubini type extensions for minimal rational curves are proven for local holomorphic maps which are only assumed to preserve the \textit{tangent directions} of minimal rational curves. Proving such maps actually respect minimal rational curves  is a non-trivial part of the extension. Now if minimal rational curves are replaced by more general $Q$-cycles, we apply the results obtained by Mok-Zhang~\cite{MokZhang} related to the so-called \textit{Schubert rigidity} on rational homogeneous spaces (see~\cite{MokZhang} for more background), which allow us to show that the preservation of tangencies imply the $Q$-cycle-respecting property. The detail will be given in Section~\ref{last section} and the following extension is what we can obtain for $Q$-cycles, which parallels the usual Cartan-Fubini type extension for minimal rational curves.

\begin{theorem} \label{cartan-fubini}
Let $G/P$ be a rational homogeneous space of Picard number 1 and $Q\subsetneq G$ be a parabolic subgroup such that $Q\not\subset P$ and $P\cap Q$ is parabolic. If $U\subset G/P$ is a connected open set and $f: U \rightarrow f(U)\subset G/P$ is a biholomorphism such that it sends the tangent space of any $Q$-cycle to the tangent space of some $Q$-cycle, then $f$ extends to a biholomorphism of $G/P$.
 \end{theorem}

\section{$Q$-cycles, $Q$-towers and sheaf of $Q$-cycle-respecting maps}

\subsection{$Q$-cycles on $G/P$}$\,$

Let $P, Q\subset G$ be parabolic subgroups of a complex simple Lie group $G$ such that $P\cap Q$ is parabolic. Then we have a double fibration
\begin{eqnarray*}
\xymatrix{
  &G/(P\cap Q)\ar[dl]_{\bf q} \ar[dr]^{\bf p}  \\
     G/Q&&G/P
 }
\end{eqnarray*}

Through the double fibration, any point $s\in G/Q$ defines a rational homogeneous subspace $\mathcal P_s:={\bf p}({\bf q}^{-1}(s))$ in $G/P$, which we will call a {\it $Q$-cycle} on $G/P$. If we consider the canonical left action of $Q$ on $G/P$, then any $Q$-cycle is just a $G$-translate of the $Q$-orbit $Qx_0\cong Q/(P\cap Q)$, where $x_0\in G/P$ is the point where $P$ is the isotropy group.
Similarly, any point $x\in G/P$ gives a rational homogeneous subspace $\mathcal Q_x:={\bf q}({\bf p}^{-1}(x))$ in $G/Q$, called a {\it $P$-cycle} on $G/Q$.

\noindent\textbf{Remark.} We have allowed the cycles to be zero-dimensional but only the positive dimensional ones are relevant in the current article.

\begin{definition}
Let $k\in\Bbb N$. We call a $k$-tuple $(\mathcal P_{s_1},\ldots, \mathcal P_{s_k})$ \textit{a chain of $Q$-cycles of length $k$}, or simply \textit{a $k$-chain of $Q$-cycles}, if $\mathcal P_{s_i}\cap\mathcal P_{s_{i+1}}\neq\emptyset$ for $1\leq i\leq k-1$.
\end{definition}

Since $P\cap Q$ is parabolic, there exists a Cartan subalgebra $\mathfrak h\subset\mathfrak g$ and a choice of simple roots $\Delta$ for $\mathfrak h$ such that $P$, $Q$ correspond to two different markings on the Dynkin diagram $\mathcal D(G)$ of $G$. Let $\psi_P$ and $\psi_Q$ be the set of marked nodes associated to $P$, $Q$ respectively. 
For a subset $\chi$ of $\Delta$ we say that $\chi$ {\it separates} $\psi_P$ and $\psi_Q$ if any connected subdiagram of the Dynkin diagram $\mathcal D(G)$ of $G$ containing both a node in $\psi_P$ and a node in $\psi_Q$ also contains a node in $\chi$. The smallest subset $\psi_Q'$ of $\psi_Q$ which separates $\psi_P$ and $\psi_Q$ is called the {\it reduction} of $\psi_Q$ mod $\psi_P$.
%
The parabolic subgroup $Q'$ corresponding to $\psi_{Q'}$ is called the \textit{reduction of $Q$ mod $P$}.

More generally, let $G$ be a complex semisimple Lie group having a decomposition $G=G_1\times\cdots\times G_m$ into simple factors and $P,Q\subset G$ be parabolic subgroups such that $P\cap Q$ is parabolic. Then we can write $P=P_1\times\cdots\times P_m$ and $Q=Q_1\times\cdots\times Q_m$ such that $P_k,Q_k$ and $P_k\cap Q_k$ are parabolic subgroups of $G_k$ for every $k$. In this case, if $Q'_k$ is the reduction of $Q_k$ mod $P_k$ for all $k$, then $Q':=Q'_1\times\cdots\times Q'_m$ is called the reduction of $Q$ mod $P$. We say that $Q$ is \textit{reduced mod} $P$ if $Q=Q'$.

\begin{proposition} [\cite{T55}, Corollary of Theorem 3] \label{parameter space}
 Let $G$ be a complex semisimple Lie group   which is a direct product of simple factors. Let $P, Q\subset G$ be parabolic subgroups such that $P\cap Q$ is parabolic
and let $Q'$ be the reduction of $Q$ mod $P$. The moduli space of $Q$-cycles on $G/P$ is $G/Q'$. More precisely, if we let $x_0\in G/P$ be the point whose isotropy group is $P$, then we have $Q'=\{g\in G: g Qx_0=Qx_0\}$.
\end{proposition}

If $Q$ is reduced mod $P$ and vice versa, we have the following interpretation of $G/(P\cap Q)$ which will be important for establishing our extension theorems. In what follows, we will denote by $Gr(d,T(M))$ the Grassmannian bundle of $d$-dimensional holomorphic tangent subspaces of a complex manifold $M$. In addition, for the sake of making a simpler statement, we take $Gr(0,T(M)):=M$.

\begin{proposition}\label{identification}
Let $G$ be a complex semisimple Lie group which is a direct product of simple factors. Let $P, Q\subset G$ be parabolic subgroups such that $P\cap Q$ is parabolic. Suppose $Q$ is reduced mod $P$ (resp. $P$ is reduced mod $Q$). Then $G/(P\cap Q)$ can be holomorphically embedded as a closed complex submanifold in $Gr(k, T(G/P))$, where $k=\dim_\mathbb C(\mathcal P_s)$ for any $s\in G/Q$ (resp. $Gr(\ell,T(G/Q))$, where $\ell=\dim_\mathbb C(\mathcal Q_x)$ for any $x\in G/P$).
\end{proposition}
\begin{proof}
Suppose $Q$ is reduced mod $P$. We just need to consider the non-trivial cases where $\dim_{\mathbb C}(\mathcal P_s)=k>0$.
We can regard $G/(P\cap Q)$ as the universal family of $Q$-cycles on $G/P$. Let $\zeta\in G/(P\cap Q)$ and $x:={\bf p}(\zeta)$, $s:={\bf q}(\zeta)$. Then by sending $\zeta$ to the holomorphic tangent space of $\mathcal P_{s}$ at $x$, we get a holomorphic map $h:G/(P\cap Q)\rightarrow Gr(k,T(G/P))$. Note that the $Q$-cycles are compact and the moduli of $Q$-cycles passing through a given point in $G/P$, which is just a $P$-cycle in $G/Q$ is also compact. It then follows that the holomorphic tangent space of a $Q$-cycle at any of its point already determines the cycle itself. (See for example~\cite{HongMok}, Lemma 3.4 for a proof of this.) Thus, $h$ is a holomorphic embedding. Similarly, $G/(P\cap Q)$ can be embedded in $Gr(\ell,T(G/Q))$ if $P$ is reduced mod $Q$.
\end{proof}


\subsection{$Q$-towers on $G/P$}\label{towersection}$\,$

We are now going to construct some projective manifolds, called the $Q$-towers, associated to $G/P$ which are essentially some kind of universal families of the $k$-chains of $Q$-cycles. In order to simplify the notations, we write $\mathcal U:=G/(P\cap Q)$ and denote the canonical projections by ${\bf p}:\mathcal U\rightarrow G/P$ and ${\bf q}:\mathcal U\rightarrow G/Q$. We will define the $Q$-towers recursively. First of all, let the restriction of $\bf q$ to ${\bf p}^{-1}(x_0)$ be ${\bf q}_1$. Consider the pullback of the bundle ${\bf q}:\mathcal U\rightarrow G/Q$ by ${\bf q}_1$ and denote it by ${\bf\hat q}_1:{\bf q}^*_1\mathcal U\rightarrow{\bf p}^{-1}(x_0)$. Associated to the pullback, there is a map ${\bf\widetilde q}_1:{\bf q}^*_1\mathcal U\rightarrow\mathcal U$ such that ${\bf q}\circ{\bf\widetilde q}_1={\bf q}_1\circ{\bf\hat q}_1$. To summarize, we have now
$$
\xymatrix
{
  {\bf q}_1^*\mathcal U \ar[d]^{{\bf\hat q}_1}\ar[r]^{{\bf\widetilde q}_1}&\mathcal U\ar[d]^{\bf q}\\
     {\bf p}^{-1}(x_0)\ar[r]^{{\bf q}_1}&G/Q
 }
$$

We call $\mathcal T^1_{x_0}:={\bf q}^*_1\mathcal U$ the first $Q$-tower at $x_0$. Let ${\bf p}_1:={\bf p}\circ{\bf\widetilde q}_1$. Geometrically, the image ${\bf p}_1(\mathcal T^1_{x_0})\subset G/P$ is the union of all the $Q$-cycles passing through $x_0$ and ${\bf p}_1$ is just the \textit{evaluation map} when we regard ${\bf\hat q}_1:\mathcal T^1_{x_0}\rightarrow{\bf p}^{-1}(x_0)$ as the universal family of $Q$-cycles passing through $x_0$.

To construct the second $Q$-tower, consider the pullback of the bundle ${\bf p}:\mathcal U\rightarrow G/P$ by ${\bf p}_1:\mathcal T^1_{x_0}\rightarrow G/P$ and denote it by ${\bf\hat p}_1:{\bf p}_1^*\mathcal U\rightarrow\mathcal T^1_{x_0}$. Thus, we have the commutative diagram
$$
\xymatrix
{
  {\bf p}_1^*\mathcal U \ar[d]^{{\bf\hat p}_1}\ar[r]^{{\bf\widetilde p}_1}&\mathcal U\ar[d]^{\bf p}\\
     \mathcal T^1_{x_0}\ar[r]^{{\bf p}_1}&G/P
 }
$$
for some ${\bf\widetilde p}_1:{\bf p}_1^*\mathcal U\rightarrow\mathcal U$.
Now write ${\bf q}_2:{\bf p}_1^*\mathcal U\rightarrow G/Q$, where ${\bf q}_2:={\bf q}\circ{\bf\widetilde p}_1$ and let ${\bf\hat q}_2:{\bf q}_2^*\mathcal U\rightarrow {\bf p}_1^*\mathcal U$ be the pullback of ${\bf q}:\mathcal U\rightarrow G/Q$ by ${\bf q}_2$. We now have the commutative diagram
$$
\xymatrix
{
  {\bf q}_2^*\mathcal U \ar[d]^{{\bf\hat q}_2}\ar[r]^{{\bf\widetilde q}_2}&\mathcal U\ar[d]^{\bf q}\\
     {\bf p}_1^*\mathcal U\ar[r]^{{\bf q}_2}&G/Q
 }
$$
for some ${\bf\widetilde q}_2:{\bf q}_2^*\mathcal U\rightarrow\mathcal U$.
Then the second $Q$-tower at $x_0$ is $\mathcal T^2_{x_0}:={\bf q}_2^*\mathcal U$. It is equipped with the evaluation map ${\bf p}_2: {\bf q}_2^*\mathcal U\rightarrow G/P$, where ${\bf p}_2={\bf p}\circ{\bf\widetilde q}_2$.

Assume now for some $k\geq 2$, the $k$-th $Q$-tower $\mathcal T^k_{x_0}$ together with the evaluation map ${\bf p}_k:\mathcal T^k_{x_0}\rightarrow G/P$ have been constructed. Consider the pullback ${\bf\hat p}_k:{\bf p}_k^*\mathcal U\rightarrow\mathcal T^k_{x_0}$ of the bundle ${\bf p}:\mathcal U\rightarrow G/P$ by ${\bf p}_k$ and let ${\bf\widetilde p}_k:{\bf p}_k^*\mathcal U\rightarrow\mathcal U$ be the map such that ${\bf p}\circ{\bf\widetilde p}_k={\bf p}_k\circ{\bf\hat p}_k$.

Write ${\bf q}_{k+1}:{\bf p}_k^*\mathcal U\rightarrow G/Q$, where ${\bf q}_{k+1}:={\bf q}\circ{\bf\widetilde p}_k$ and let ${\bf\hat q}_{k+1}:{\bf q}_{k+1}^*\mathcal U\rightarrow {\bf p}_k^*\mathcal U$ be the pullback of ${\bf q}:\mathcal U\rightarrow G/Q$ by ${\bf q}_{k+1}$, together with the map ${\bf\widetilde q}_{k+1}:{\bf q}_{k+1}^*\mathcal U\rightarrow\mathcal U$. The $(k+1)$-th $Q$-tower at $x_0$ is $\mathcal T^{k+1}_{x_0}:={\bf q}_{k+1}^*\mathcal U$, equipped with the evaluation map ${\bf p}_{k+1}:{\bf q}_{k+1}^*\mathcal U\rightarrow G/P$, where ${\bf p}_{k+1}={\bf p}\circ{\bf\widetilde q}_{k+1}$.

As for the first $Q$-tower,   the image of ${\bf p}_k(\mathcal T^k_{x_0})$ is just the union of the images of all the $k$-chains $(\mathcal P_{s_1},\ldots,\mathcal P_{s_k})$ of $Q$-cycles such that $x_0\in\mathcal P_{s_1}$.

There is an alternative description for the $Q$-towers, as follows. By keeping track of each pullback procedure during the construction, it is not difficult to see that $\mathcal T^k_{x_0}$ can be realized as the closed complex submanifold of $(G/P)^k\times (G/Q)^k$ consisting of the $2k$-tuples $(x_1,\ldots,x_k, s_1,\ldots,s_k)$ such that for every $j\in\{1,\ldots, k\}$, the $Q$-cycle $\mathcal P_{s_j}$ contains both $x_{j-1}$ and $x_j$. In this way, we see that $\mathcal T^k_{x_0}$ is a projective manifold. The evaluation map ${\bf p}_k$ is just the projection $(x_1,\ldots,x_k,s_1,\ldots,s_k)\mapsto x_k$.
Furthermore, for $1\leq j<k$, the projection $\pi_{k,j}$ defined by
$$
\pi_{k,j}(x_1,\ldots,x_k,s_1,\ldots,s_k) := (x_1,\ldots,x_j,s_1,\ldots,s_j)
$$
is clearly a holomorphic surjection from $\mathcal T^k_{x_0}$ to $\mathcal T^j_{x_0}$.


It follows immediately from the definition that $\pi_{k,j}^{-1}(x_1,\ldots,x_j,s_1,\ldots,s_j)$ can be canonically identified with $\mathcal T^{k-j}_{x_j}$. Since $G/P$ is homogeneous and its automorphisms respect the $Q$-cycles, we have the biholomorphism $\mathcal T^\ell_x\cong\mathcal T^\ell_{x_0}$ for every $\ell\in\mathbb N^+$ and  $x\in G/P$. Thus, whenever $1\leq j <k$, we have the holomorphic fiber bundle
$$
\mathcal T^{k-j}_{x_0}\longrightarrow\mathcal T^k_{x_0}\overset{\pi_{k,j}}{\longrightarrow}\mathcal T^j_{x_0}.
$$

\begin{definition}\label{diagonal}
Let $\mathcal D^k_{x_0}=\{(x_1,\ldots,x_k,s_1,\ldots,s_k)\in\mathcal T^k_{x_0}: x_1=\cdots=x_k=x_0\}$. We call $\mathcal D^k_{x_0}$ the \textit{diagonal preimage} of $x_0$ in $\mathcal T^k_{x_0}$.
It is obvious that $\mathcal D^1_{x_0}={\bf p}_1^{-1}(x_0)$, $\mathcal D^k_{x_0}\subset {\bf p}_k^{-1}(x_0)$ and $\pi_{k,j}(\mathcal D^k_{x_0})=\mathcal D^j_{x_0}$ whenever $1\leq j <k$.
\end{definition}

\subsection{$Q$-cycle-connectivity}$\,$

Using   $Q$-cycles, we can define an equivalence relation $\sim$ on $G/P$ as follows. For any pair of two points $x,y \in G/P$, we say that $x\sim y$ if  there is a chain of $Q$-cycles $(\mathcal P_{s_1},\ldots,\mathcal P_{s_k})$ such that $x\in\mathcal P_{s_1}$ and $y\in\mathcal P_{s_k}$. Then the equivalence relation $\sim$ is $G$-equivariant, i.e. $x \sim y$ if and only if $gx \sim gy$ for any $g \in G$. Thus there is a holomorphic surjective map from $G/P$ to the space of equivalence classes of $\sim$ (which is also a rational homogeneous space of $G$).

\begin{definition}
We say that $G/P$ is \textit{$Q$-cycle-connected} if for every $x, y\in G/P$, there is a chain of $Q$-cycles connecting $x$ and $y$, i.e. $x\sim y$.
\end{definition}

\begin{proposition}\label{connectivity}
Let $P, Q\subset G$ be parabolic subgroups such that $P\cap Q$ is parabolic. Then the following are equivalent:

(1) $G/P$ is $Q$-cycle-connected;

(2) There exists $N\in\mathbb N^+$ such that for every $x,y\in G/P$, there is an $N$-chain of $Q$-cycles connecting $x$ and $y$;

(3) There exists $N\in\mathbb N^+$ such that for every $k\geq N$, the evaluation map ${\bf p}_k:\mathcal T^k_{x}\rightarrow G/P$ is surjective for any $x$;

(4) There does not exist any parabolic subgroup $R\subsetneq G$ containing $P$ and $Q$.
\end{proposition}
\begin{proof}
The implication (2) $\Rightarrow$ (1) is trivial. To see that (1) implies (2), fix a base point $x_0\in G/P$ and consider the evaluation map of ${\bf p}_k:\mathcal T^k_{x_0}\rightarrow G/P$ for every $k$. Each ${\bf p}_k$ is a proper holomorphic map between irreducible projective varieties and thus ${\bf p}_k(\mathcal T^k_{x_0})$ is an irreducible algebraic subvariety of $G/P$. Recall that ${\bf p}_k(\mathcal T^k_{x_0})$ is the union of the images of all the $k$-chains of $Q$-cycles emanating from $x_0$ and in particular, ${\bf p}_k(\mathcal T^k_{x_0})\subset{\bf p}_{k+1}(\mathcal T^{k+1}_{x_0})$. Therefore, by considering the dimension, there is an $N\in\mathbb N^+$ such that ${\bf p}_k(\mathcal T^k_{x_0})={\bf p}_N(\mathcal T^N_{x_0})$ for every $k\geq N$. But (1) implies that ${\bf p}_N(\mathcal T^N_{x_0})$ cannot be a proper subset of $G/P$ and hence we have  ${\bf p}_k(\mathcal T^k_{x_0})={\bf p}_N(\mathcal T^N_{x_0})=G/P$ for every $k\geq N$. The previous argument also demonstrates that (1) $\Rightarrow$ (3) and the converse is trivial.

Suppose now $G/P$ is not $Q$-cycle-connected. Then as explained previously, there is a $G$-equivariant holomorphic fibration $\psi:G/P\rightarrow G/R$ for some parabolic subgroup $R\subsetneq G$ containing $P$. Now let $x_0\in G/P$ be the point whose isotropy group is just $P$. Then the $Q$-cycle $Qx_0$ is contained in $\psi^{-1}(\psi(x_0))=Rx_0$ by our construction of $G/R$. Thus, for every $q\in Q$, there exists $r\in R$ such that $qx_0=rx_0$ and hence $q^{-1}r\in P\subset R$. Therefore,  $q\in R$ and it follows that $Q\subset R$. We have thus established (4) $\Rightarrow$ (1). Finally, if there exists a parabolic subgroup $R\subsetneq G$ containing $P$ and $Q$, then any $Q$-cycle is contracted to a point by the $G$-equivariant projection $G/P\rightarrow G/R$ and it is immediate that $G/P$ is not $Q$-cycle-connected. Hence, (1) $\Rightarrow$ (4).
\end{proof}

\noindent\textbf{Remark.} Regarding the cycle-connectivity on complex homogeneous manifolds, we noted that a related but different statement is given by Koll\'ar~(\cite{K15}, Theorem 2).

\begin{corollary}
If $G/P$ is of Picard number 1, then (1) to (4) in Proposition~\ref{connectivity} hold for whenever $Q\not\subset P$, or equivalently, whenever the $Q$-cycles are of positive dimension.
\end{corollary}
\begin{proof}
The condition (4) in Proposition~\ref{connectivity} obviously holds since $P$ is a maximal parabolic subgroup.
\end{proof}

\begin{lemma}\label{simply connectedness}
The canonical left action of $P$ on $G/P$ has a unique open orbit. Furthermore, the open orbit is simply connected.
\end{lemma}
\begin{proof}
As $P$ is parabolic, it only has finitely many orbits in $G/P$~(e.g. see \cite{Bor}). The irreducibility (as an algebraic variety) of $G/P$ implies that the union of all the open orbit(s) is connected since it is the complement of the union of the lower dimensional orbits and the latter is a proper complex algebraic subvariety $\mathcal Z\subset G/P$. Therefore, there is just one open orbit and we denote it by $\mathcal O$.

To show that $\mathcal O$ is simply connected, we first consider the case when $G$ is simple. Fix a Borel subgroup $B$ contained in $P$. Then $\mathcal Z$ is a union of Schubert varieties of $B$. Now consider the involution $\iota$ of the Dynkin diagram of $G$ defined by mapping a simple root $\alpha$ to $-\omega_0(\alpha)$, where $\omega_0$ is the longest element of the Weyl group of $G$. The involution $\iota$ is nontrivial if $G$ is of type $A_{\ell}$, $D_{\ell}$, and  $E_6$, and is the identity otherwise.  We will divide into two cases: (1) when $\phi_P$ is invariant under $\iota$;  (2) when $\phi_P \not=\iota(\phi_P)$.

When $\phi_P=\iota(\phi_P)$, there is a point $x_{\infty}$ in $G/P$ whose isotropy group is the opposite parabolic subgroup $P^-$ (just take $x_{\infty} :=\omega_0.x_0$, where $x_0$ is the point in $G/P$ at which the isotropy group is $P$).  Let  $U$ be the unipotent part of $P$. Then the open orbit $\mathcal O$ is the $P$-orbit $P.x_{\infty}=U.x_{\infty}$ and is biholomorphic to $  \mathbb C^n$, where $n=\dim_{\mathbb C}G/P$.

When $\phi_P\not=\iota(\phi_P)$, then   $\mathcal Z$   has codimension $\geq 2$ if  $\phi_P \cap \iota(\phi_P)  =\emptyset$, and has codimension one if  $\phi_P \cap \iota(\phi_P) \not=\emptyset$  (Proposition 6 of \cite{P02}). In the first case, $\mathcal O$ is simply connected because it is the complement of a complex subvariety of codimension at least 2 in $G/P$ which is simply connected. In the second  case, let $\phi'$ be a subset of $\phi_P$ such that $\phi' \cap \iota(\phi_P) =\emptyset$ (just take $\phi':=\phi_P -\iota(\phi_P)$). Since $\phi_P \not=\iota(\phi_P)$,  $\phi' $ is nonempty. Let $P'$ be the corresponding parabolic subgroup of $G$ and $\tau:G/P \rightarrow G/P'$ be the  projection map. Then the complement of the open $P$-orbit $\mathcal O'$ in $G/P'$ has codimension at least 2 in $G/P'$  (Proposition 6 of \cite{P02}) and thus $\mathcal O'$ is simply connected. Since $\tau|_{\mathcal O}: \mathcal O \rightarrow \mathcal O'$ is a surjective map whose fibers are connected, $\mathcal O$ is simply connected.

Finally, when $G$ is a direct product simple complex Lie groups, then $\mathcal O$ decomposes as a product correspondingly. Each factor of the product is simply connected by the argument above and thus their product $\mathcal O$ is simply connected.
\end{proof}

\begin{proposition}\label{fiber connectedness}
If $N$ is a positive integer such that any $x\in G/P$ can be connected to $x_0$ by a chain of $Q$-cycles of length at most $N$, then for every $n\geq N$, the fibers of the evaluation map ${\bf p}_n:\mathcal T^n_{x_0}\rightarrow G/P$ are connected.
\end{proposition}
\begin{proof}
Fix a positive integer $n \geq N$ and let $\mathcal T:=\mathcal T^n_{x_0}$,
 $\rho:={\mathbf p}_n$.
Then $\rho: \mathcal T \rightarrow G/P$ is a holomorphic surjection between two projective manifolds and in particular, is proper. By Stein factorization, there are a projective variety $Y$ and  a holomorphic map $\rho': \mathcal T \rightarrow Y$ with connected fibers  and a finite holomorphic map $\phi: Y \rightarrow G/P$ such that $\rho=\phi \circ \rho'$. It suffices to show that $\phi$ is actually biholomorphic.

By the homogeneity of $G/P$ we can assume that the action of $P$ on $G/P$ fixes $x_0$. Since the double fibration $G/Q \leftarrow \mathcal U \rightarrow G/P$ is $G$-equivariant, we see that $P$ acts on $\mathcal T$ and $\rho$ is $P$-equivariant. Recall that in the construction of Stein factorization, the variety $Y$ is just the space of connected components of the fibers of $\rho$. Thus, $Y$ is acted by $P$ accordingly and in the factorization both $\phi$ and $\rho'$ are $P$-equivariant.


As mentioned in the proof of Lemma~\ref{simply connectedness}, there are only finitely many $P$-orbits in $G/P$. Consequently, $Y$ also has only finitely many $P$-orbits since $\phi:Y \rightarrow G/P$ is a finite map. Thus, $Y$ has an open $P$-orbit.  As $\mathcal T$ is irreducible and hence so is $Y$, the union of the open orbit(s) of $Y$ is connected since the union of the lower dimensional orbits is a proper complex algebraic subvariety. Therefore, there is just one open $P$-orbit on $Y$ and we denote it by $\mathcal O_Y$.
By Lemma~\ref{simply connectedness}, the unique open $P$-orbit $\mathcal O\subset G/P$ is simply connected. Now the $P$-equivariant map $\phi$ maps $\mathcal O_Y$ onto $\mathcal O$. Since $\mathcal O_Y$ and $\mathcal O$ are $P$-orbits, the finite map $\phi|_{\mathcal O_Y}$ is necessarily unramified since the ramification locus must be $P$-invariant.
Thus, $\phi|_{\mathcal O_Y} :\mathcal O_Y \rightarrow \mathcal O$ is a biholomorphism and hence $\phi:Y\rightarrow G/P$ is a finite birational holomorphic map onto $G/P$. Note that $G/P$ is smooth and in particular, is a normal variety. By Zariski Main Theorem~(\cite{Hart} Corollary III.11.4), the inverse image of $\phi$ of every point of $G/P$ is connected and consequently, $\phi$ is a biholomorphism.
\end{proof}

\subsection{Sheaf of $Q$-cycle-respecting maps $\mathscr B$}\label{sheaf} $\,$

Let $f:U\subset G/P\rightarrow G/P$ be a local biholomorphism defined on an open set $U$. We say that $f$ \textit{is $Q$-cycle-respecting} if for every $Q$-cycle $\mathcal P_s$ intersecting $U$, the image of each connected component of $\mathcal P_s\cap U$ is contained in some $Q$-cycle $\mathcal P_{s'}$. If we assign to every open set $U\subset G/P$ the set of such local biholomorphisms, it is easily seen that we get a presheaf $\mathscr B^o$ on $G/P$ and we will denote the sheaf associated to $\mathscr B^o$ by $\mathscr B$.

We call $\mathscr B$ the \textit{sheaf of $Q$-cycle-respecting maps}. The set of sections of $\mathscr B$ on $U$   will be denoted by $\mathscr B(U)$ and for $\mathfrak S\in\mathscr B(U)$, its germ at $x\in U$ will be denoted by $\mathfrak S_x$. We will denote by $\mathscr B_x$ the stalk of $\mathscr B$ at $x$. By our definition, for an open set $U\subset G/P$, the sections in $\mathscr B(U)$ correspond to the holomorphic maps from $U$ to $G/P$ that are locally $Q$-cycle-respecting biholomorphism. If $\mathfrak f\in\mathscr B_x$ and it is represented by a locally $Q$-cycle-respecting map $f:U\rightarrow G/P$, then it is clear that the point $f(x)$ only depends on $\mathfrak f$ and we will denote it by $\mathfrak f(x)$.

The following proposition says essentially that for a given germ in $\mathscr B$, we can choose a local representative of it having a ``stronger" $Q$-cycle-respecting behavior.

\begin{proposition}\label{good neighborhood}
Let $f:U \subset G/P \rightarrow G/P$ be a $Q$-cycle-respecting local biholomorphism defined on an open set $U$.
There exists an open set $V_1\subset U$ such that for any $Q$-cycle $\mathcal P_s$ intersecting $V_1$, we have $f(\mathcal P_s\cap V_1)\subset\mathcal P_{s'}$ for some $s'\in G/Q$ and the analogous condition also holds for $f^{-1}|_{V_2}$, where $V_2=f(V_1)$.
\end{proposition}

\begin{proof} Let $L$ be an ample line bundle on $G/P$. Then $L$ is very ample (Section 2.8 of \cite{BL}).  Consider the embedding of $G/P$ into $\mathbb P(V)$ by $L$. Let $\mathcal P_0$ be the $Q$-orbit  of $x_0$. Then  $\mathcal P_0$ is the linear section $G/P \cap \mathbb P(V_0)$ of $G/P$ by $\mathbb P(V_0)$, where $\mathbb P(V_0)$ be the linear span of $\mathcal P_0$ in $\mathbb P(V)$ (Section 2.10 and 2.11 of \cite{BL}). Here we use the fact that $\mathcal P_0$ is a Schubert variety of $G/P$.   Furthermore, since $G$ acts on $\mathbb P(V)$ linearly, $g\mathbb P(V_0)$ are linear spaces in $\mathbb P(V)$ and $g\mathcal P_0$ are linear sections $ G/P\cap  g\mathbb P(V_0) $.
Take a neighbourhood $U'$ of $x_0$ in $U$ which is convex in the sense that $ g\mathbb P(V_0) \cap U'$ are connected for any $g \in G$. Then $g \mathcal P \cap U' = g \mathbb P(V_0) \cap U'$ is connected for any $g \in G$. By the $Q$-cycle-respecting property of $f$, for any $Q$-cycle $\mathcal P_s$ intersecting $U'$, we have $f(\mathcal P_s \cap U') \subset \mathcal P_{s'}$ for some $s' \in G/Q$.

To get the analogous condition for $f^{-1}$, take a convex neighborhood $V''$ of $f(x_0)$ in $V':=f(U')$ in the same way as above and put $V_1:=f^{-1}(V'')$. Then we still have the property that   for any $Q$-cycle $\mathcal P_s$ intersecting $V_1$, $\mathcal P_s \cap V_1 \subset \mathcal P_{s'}$ for some $s' \in G/Q$ (note. now the intersection $\mathcal P_s \cap V_1$ may be disconnected).  By the convexity of $V''$,  the analogous condition holds for $f^{-1}$ on $V_2=f(V_1) =V''$.
\end{proof}

For our purposes, given a sheaf $\mathscr S$ on a complex manifold $X$, we will equivalently regard a section $\mathfrak S\in\mathscr S(U)$ on an open set $U\subset X$ as a continuous map from $U$ to the \textit{espace \'etal\'e}~\cite{Hart} of $\mathscr S$, defined by $\mathfrak S(x)=\mathfrak S_x$, $x\in U$. Here we recall that the espace \'etal\'e of $\mathscr S$ is the set $\bigcup_{x\in X}\mathscr S_x$, equipped with the topology generated by the base $\{\mathfrak S(U): \mathfrak S\in\mathscr S(U), U\subset X\}$, where $\mathfrak S(U)$ is the image of $U$ under the map $\mathfrak S$ just described.

The espace \'etal\'e of $\mathscr B$ is a Hausdorff topological space by the identity theorem for holomorphic functions. In particular, for a continuous curve $\Gamma:[0,1]\rightarrow G/P$ and a germ $\mathfrak f\in\mathscr B_{\Gamma(0)}$, there is at most one lifting $\tilde\Gamma$ to the espace \'etal\'e of $\mathscr B$ such that $\tilde\Gamma(0)=\mathfrak f$. If such lifting exists, we call $\tilde\Gamma(1)\in\mathscr B_{\Gamma(1)}$ the \textit{analytic continuation} of $\mathfrak f$ along $\Gamma$.

Now let $M$ be a complex manifold and $g:M\rightarrow G/P$ be a holomorphic map. The \textit{inverse image} $g^{-1}\mathscr B$ of $\mathscr B$ on $M$ is by definition~\cite{Hart} the sheaf associated to the presheaf assigning the direct limit $\displaystyle\varinjlim_{U\supseteq g(V)}\mathscr B(U)$ to an open set $V\subset M$, where the limit is taken over all the open sets $U\subset G/P$ containing $g(V)$. Recall that for $y\in M$, there is a canonical identification $(g^{-1}\mathscr B)_y\cong\mathscr B_{g(y)}$. If we let $\mathscr E$ be the espace \'etal\'e of $\mathscr B$ and $\pi:\mathscr E\rightarrow G/P$ be the canonical projection, then the espace \'etal\'e of $g^{-1}\mathscr B$ is just the pullback of $\pi:\mathscr E\rightarrow G/P$ by $g$ and we have the commutative diagram:
$$
\xymatrix
{
  g^{-1}\mathscr E \ar[d]^{\pi_M}\ar[r]^{\tilde g}&\mathscr E\ar[d]^{\pi}\\
     M\ar[r]^{g}&G/P
 }
$$
where $g^{-1}\mathscr E:=\{(m,\mathfrak e)\in M\times\mathscr E: g(m)=\pi(\mathfrak e)\}\cong\bigsqcup_{m\in M}\mathscr B_{g(m)}$ is equipped with the subspace topology from $M\times\mathscr E$, and $\pi_M$, $\tilde g$ are the projections to $M$ and $\mathscr E$ respectively.
We can define the notion of analytic continuation for $g^{-1}\mathscr B$ analogously and it is related to the analytic continuation of $\mathscr B$ as follows.

\begin{proposition}\label{relating analytic continuation}
Let $\Gamma:[0,1]\rightarrow M$ be a continuous curve and $\mathfrak f_0\in (g^{-1}\mathscr B)_{\Gamma(0)}$. Suppose that $\mathfrak f_1\in (g^{-1}\mathscr B)_{\Gamma(1)}$ is an analytic continuation of $\mathfrak f_0$ along $\Gamma$, then through the canonical identifications $\tilde g:(g^{-1}\mathscr B)_{\Gamma(t)}\overset{\cong}\rightarrow\mathscr B_{g(\Gamma(t))}$, $t\in[0,1]$, the germ $\mathfrak f_1\in\mathscr B_{g(\Gamma(1))}$ is also the analytic continuation of $\mathfrak f_0\in\mathscr B_{g(\Gamma(0))}$ along $g\circ\Gamma$.
\end{proposition}
\begin{proof}
It follows immediately from the commutative diagram above.
\end{proof}

\begin{corollary}\label{constant on fibers}
Let $i:U\hookrightarrow M$ be the inclusion of an open set $U$ and regard a section $\mathfrak S\in (g^{-1}\mathscr B)(U)$ as a continuous map $\mathfrak S:U\rightarrow g^{-1}\mathscr E$ such that $\pi_M\circ\mathfrak S|_U=i$. If $Z\subset U$ is a connected set such that $g|_Z$ is a constant, then $\tilde g\circ\mathfrak S|_Z$ is a constant.
\end{corollary}
\begin{proof}
Let $g|_Z\equiv x\in G/P$. Then, by the commutative diagram above, $\tilde g\circ\mathfrak S(Z)$ is a connected set in $\pi^{-1}(x)$. But by definition $\pi^{-1}(x)$ is a discrete set in $\mathscr E$ and hence $\tilde g\circ\mathfrak S|_Z$ is a constant.
\end{proof}

\section{Extension for germs of $Q$-cycle-respecting maps}

\subsection{Local extension}$\,$

Fix an arbitrary base point $x_0\in G/P$ and $\mathfrak f\in\mathscr B_{x_0}$.

We will first establish a local version of the extension. In what follows, we will adopt the notations in Section~\ref{towersection} and write $\mathcal U:=G/(P\cap Q)$, and ${\bf p}:\mathcal U \rightarrow G/P$ and ${\bf q}:\mathcal U\rightarrow G/Q$ for the canonical projections.
Recall that for $x_0\in G/P$, the set $\mathcal Q_{x_0}={\bf q}({\bf p}^{-1}(x_0))$ contains precisely the points $s\in G/Q$ such that $x_0\in\mathcal P_s$.

\begin{proposition}\label{firstextension}
Suppose $P$ is reduced mod $Q$ and  vice versa. Let $x_0\in G/P$ and $\mathfrak f\in\mathscr B_{x_0}$ be a germ of $Q$-cycle-respecting map.  There exist connected open sets $\mathcal V_1, \mathcal V_2\subset\mathcal U:=G/(P\cap Q)$ containing ${\bf q}^{-1}(\mathcal Q_{x_0})$, ${\bf q}^{-1}(\mathcal Q_{\mathfrak f(x_0)})$ respectively, and a biholomorphic map $\mathcal F:\mathcal V_1\rightarrow \mathcal V_2$, a section $\mathfrak F\in ({\bf p}^{-1}\mathscr B)(\mathcal V_1)$ such that,
under the canonical identification $({\bf p}^{-1}\mathscr B)_{\zeta}\cong\mathscr B_{{\bf p}(\zeta)}$ for $\zeta\in\mathcal V_1$, we have (i) ${\bf p}(\mathcal F(\zeta))=\mathfrak F_\zeta({\bf p}(\zeta))$ for every $\zeta\in\mathcal V_1$; and (ii) $\mathfrak F_\xi=\mathfrak f$ for every $\xi\in{\bf p}^{-1}(x_0)$.
\end{proposition}
\begin{proof}
Since $Q$ is reduced mod $P$, by Proposition~\ref{identification}, $\mathcal U$ can be identified with a closed complex submanifold of the Grassmannian bundle $Gr(k,T(G/P))$ of $k$-dimensional holomorphic tangent subspaces of $G/P$, where $k=\dim_{\mathbb C}\mathcal P_s$ for any $s\in G/Q$. Here, a point $\zeta\in\mathcal U$ is then identified with the holomorphic tangent space of $\mathcal P_{{\bf q}(\zeta)}$ at ${\bf p}(\zeta)$.

By Proposition~\ref{good neighborhood}, we can choose a local biholomorphism $f:V_1\rightarrow V_2$, representing the germ $\mathfrak f\in\mathscr B_{x_0}$, where $V_1$ is a neighborhood containing $x_0$ and $V_2=f(V_1)$, satisfying the following condition:

\noindent($\dagger$) For any $Q$-cycle $\mathcal P_s$ intersecting $V_1$, we have $f(\mathcal P_s\cap V_1)\subset\mathcal P_{s'}$ for some $s'\in G/Q$ and the analogous condition also holds for $f^{-1}$ on $V_2$.


Under the identifications made on $\mathcal U$ as described above, the differential of $f$ induces a local biholomorphism on $\mathcal U$
$$
[df]:{\bf p}^{-1}(V_1)\rightarrow {\bf p}^{-1}(V_2),
$$
where ${\bf p}^{-1}(V_j)$ is regarded as a closed complex submanifold of $Gr(k,T(V_j))$ for $j=1,2$.
Now take a point $s\in G/Q$ such that $\mathcal P_s\cap V_1\neq\emptyset$. By our hypotheses, $f(\mathcal P_s\cap V_1)\subset\mathcal P_{s'}$ for some $s'\in G/Q$. This implies that $[df]$ is fiber-preserving with respect to $\bf q$. Hence, if we let $V_1^\sharp:={\bf q}({\bf p}^{-1}(V_1))$ and $V_2^\sharp:={\bf q}({\bf p}^{-1}(V_2))$, which are connected open sets on $G/Q$, then $[df]$ induces a holomorphic map $f^\sharp :V_1^\sharp\rightarrow V_2^\sharp$, which is also a biholomorphism since by the condition $(\dagger)$, it is easily seen that $[d(f^{-1})]=[df]^{-1}$.

Let $\ell=\dim_{\mathbb C}\mathcal Q_x$ for any $x\in G/P$. The differential of $f^\sharp$ gives a biholomorphism $$[df^\sharp]:Gr(\ell, T(V_1^\sharp))\rightarrow Gr(\ell, T(V_2^\sharp)),$$ where $Gr(\ell, T(V_1^\sharp))$ and $Gr(\ell, T(V_2^\sharp))$ are the restrictions of the Grassmannian bundle of $\ell$-dimensional holomorphic tangent subspaces of $G/Q$ to $V_1^\sharp$ to $V_2^\sharp$ respectively. Now observe that as $P$ is also reduced mod $Q$, we can thus, by the same token, identify $\mathcal U$ with a complex submanifold of $Gr(\ell,T(G/Q))$. Under such identification, we can regard ${\bf q}^{-1}(V_1^\sharp)\subset Gr(\ell,T(V_1^\sharp))$ and ${\bf q}^{-1}(V_2^\sharp)\subset Gr(\ell,T(V_2^\sharp))$ as closed complex submanifolds. We are going to show that $[df^\sharp]$ sends ${\bf q}^{-1}(V_1^\sharp)$ onto ${\bf q}^{-1}(V_2^\sharp)$ and moreover, $\left.[df^\sharp]\right|_{{\bf p}^{-1}(V_1)}=[df]$. Note that ${\bf p}^{-1}(V_1)$ is open in $\mathcal U$ and hence is open in ${\bf q}^{-1}(V_1^\sharp)={\bf q}^{-1}({\bf q}({\bf p}^{-1}(V_1)))$.

Take a point $\zeta\in{\bf p}^{-1}(V_1)$ and let $x:={\bf p}(\zeta)$ and $s:={\bf q}(\zeta)$. When $\zeta$ is identified with a point in $Gr(\ell, T(G/Q))$, then $\zeta$ is the holomorphic tangent space of $\mathcal Q_x$ at $s$. Through the double fibration, $\mathcal Q_x$ consists of precisely the points $t$ such that $x\in\mathcal P_t$. Since $f(\mathcal P_t\cap V_1)\subset\mathcal P_{f^\sharp(t)}$, for every $t\in\mathcal Q_x\cap V_1^\sharp$, we have $f(x)\in\mathcal P_{f^\sharp(t)}$ and hence $f^\sharp(\mathcal Q_x\cap V_1^\sharp)\subset\mathcal Q_{f(x)}$ (it is an open inclusion since $f^\sharp$ is biholomorphic). By taking the differential of $f^\sharp$, we get that $[df^\sharp](\zeta)$ is the holomorphic tangent space of $\mathcal Q_{f(x)}$ at $f^\sharp (s)$. Therefore, when $[df^\sharp](\zeta)$ is regarded as a point in $\mathcal U$, we have
$$
{\bf p}([df^\sharp](\zeta))=f(x) \textrm{\,\,\,\,\,\,and\,\,\,\,\,\,}
{\bf q}([df^\sharp](\zeta))=f^\sharp (s).
$$

Now recall that $f^\sharp$ was defined by making use of the fact that $[df]$ is fiber-preserving with respect to ${\bf q}$ and thus
$f^\sharp(s)={\bf q}([df](\zeta)).$
In addition, we obviously have ${\bf p}([df](\zeta))=f(x)$ by our identifications and the definition of $[df]$. Since $({\bf p}, {\bf q}):G/(P\cap Q)\rightarrow G/P\times G/Q$ is injective, we deduce that $[df](\zeta)=[df^\sharp](\zeta)$. Since $\zeta\in{\bf p}^{-1}(V_1)$ is arbitrary, it follows that $[df^\sharp]:{\bf q}^{-1}(V_1^\sharp)\rightarrow Gr(\ell,T(V_2^\sharp))$ is a holomorphic extension of $[df]:{\bf p}^{-1}(V_1)\rightarrow{\bf p}^{-1}(V_2)$. Furthermore, as ${\bf p}^{-1}(V_2)\subset{\bf q}^{-1}(V_2^\sharp)\subset Gr(\ell,T(V_2^\sharp))$, and the latter inclusion is an embedding of a closed complex submanifold, we conclude that we have $[df^\sharp]({\bf q}^{-1}(V_1^\sharp))\subset{\bf q}^{-1}(V_2^\sharp)$. Replacing $f$ by $f^{-1}$ and $V_1$ by $V_2$, we get a biholomorphism
$$
	[df^\sharp]:{\bf q}^{-1}(V_1^\sharp)\rightarrow{\bf q}^{-1}(V_2^\sharp).
$$

Next, we are going to construct a section in $({\bf p}^{-1}\mathscr B)({\bf q}^{-1}(V_1^\sharp))$.
Let $\zeta\in{\bf q}^{-1}(V_1^\sharp)$. Choose a connected open set $\mathcal W\subset\mathcal U$ such that $\zeta\in \mathcal W\subset{\bf q}^{-1}(V_1^\sharp)$. Recall that we have previously shown that $[df^\sharp]:{\bf q}^{-1}(V_1^\sharp)\rightarrow{\bf q}^{-1}(V_2^\sharp)$ is a holomorphic extension of $[df]:{\bf p}^{-1}(V_1)\rightarrow {\bf p}^{-1}(V_2)$. In particular, for any $x\in G/P$, ${\mathbf p} \circ [df^\sharp]$ is constant along each connected component of ${\bf p}^{-1}(x)\cap{\bf q}^{-1}(V_1^\sharp)$ whenever the intersection is non-empty. This follows from the fact that $[df]$, as the differential of $f$, satisfies the same property on ${\bf p}^{-1}(V_1)$, which is an open subset of ${\bf q}^{-1}(V_1^\sharp)$ and such analytic property is preserved in any holomorphic extension by the identity theorem for holomorphic functions. The same is true for $[df^\sharp]^{-1}$. Thus, by choosing the open set $\mathcal W$ sufficiently small, we see that $\left.[df^\sharp]\right|_{\mathcal W}:\mathcal W\rightarrow [df^\sharp](\mathcal W)$ descends to a local biholomorphism on $G/P$, which we denote by $f_{\mathcal W}$. To see that $f_{\mathcal W}$ is $Q$-cycle-respecting, it suffices to note that $f_{\mathcal W}$ descends from $[df^\sharp]$, which respects the fibers of ${\bf q}$ and this translates precisely to the condition that $f_{\mathcal W}$ is $Q$-cycle-respecting. Therefore, we deduce that $f_{\mathcal W} \in \mathscr B({\mathbf p}(\mathcal W))$. By taking the direct limit over $\mathcal W$, we then get a germ $\mathfrak f_{\zeta} \in \mathscr B_{{\mathbf p}(\zeta)}$.


Now we have obtained a map to the espace \'{e}tal\'{e} of ${\bf p}^{-1}\mathscr B$ over ${\bf q}^{-1}(V_1^\sharp)$, denoted by
$$
\mathfrak F:{\bf q}^{-1}(V_1^\sharp)\rightarrow\bigcup_{\zeta\in {\bf q}^{-1}(V_1^\sharp)}\mathscr B_{{\bf p}(\zeta)},
$$
so that $\mathfrak F(\zeta)=\mathfrak f_\zeta$. To see that $\mathfrak F$ is indeed a section in $({\bf p}^{-1}\mathscr B)({\bf q}^{-1}(V_1^\sharp))$, we take a point $\zeta\in{\bf q}^{-1}(\mathcal Q_{x_0})$ and an open neighborhood $\mathcal W\subset{\bf q}^{-1}(V_1^\sharp)$ containing $\zeta$, such that $\left.[df^\sharp]\right|_{\mathcal W}$ descends to $f_{\mathcal W}\in\mathscr B({\bf p}(\mathcal W))$, as before. Then for any $\eta\in\mathcal W$, it is by our construction that the germ of $f_{\mathcal W}$ at ${\bf p}(\eta)$ is just $\mathfrak F(\eta)=\mathfrak f_{\eta}$.
That is, on the open neighborhood $\mathcal W$ of $\zeta$, the map $\left.\mathfrak F\right|_{\mathcal W}$ is given by taking the germs of a section in $\mathscr B({\bf p}(\mathcal W))$. Hence, $\mathfrak F\in ({\bf p}^{-1}\mathscr B)({\bf q}^{-1}(V_1^\sharp))$.

Now let $\mathcal V_j:={\bf q}^{-1}(V_j^\sharp)$ for $j=1,2$ and $\mathcal F:\mathcal V_1\overset{\cong}\rightarrow\mathcal V_2$, where $\mathcal F:=[df^\sharp]|_{\mathcal V_1}$.
It remains to check that under the canonical identification $({\bf p}^{-1}\mathscr B)_{\zeta}\cong\mathscr B_{{\bf p}(\zeta)}$ for $\zeta\in\mathcal V_1$, we have
 ${\bf p}(\mathcal F(\zeta))=\mathfrak F_\zeta({\bf p}(\zeta))=\mathfrak F(\zeta)({\bf p}(\zeta))$ for every $\zeta\in\mathcal V_1$; and $\mathfrak f=\mathfrak F_\xi=\mathfrak F(\xi)$ for every $\xi\in{\bf p}^{-1}(x_0)$.

The first half is clear, it follows directly from how we construct $\mathfrak F(\zeta)$ by descending $[df^\sharp]$ locally and the fact that $\mathcal F$ is just $[df^\sharp]$. For the second half, since ${\bf p}^{-1}(V_1)$ is contained in the domain of definition of $[df]$, if $\xi\in{\bf p}^{-1}(x_0)$ and $\mathcal W$ is an open neighborhood of $\xi$ in $\mathcal U$ such that $\mathcal W\subset{\bf p}^{-1}(V_1)$, then for every $\eta\in\mathcal W$, we have $\mathfrak F(\eta)({\bf p}(\eta))={\bf p}([df^\sharp](\eta))={\bf p}([df](\eta))=f({\bf p}(\eta))$. Thus, $\mathfrak F_\xi=\mathfrak f$.
\end{proof}

\subsection{Global extension}$\,$

Recall that for every $k\in\mathbb N^+$, we have the $k$-th $Q$-tower ${\bf p}_k:\mathcal T^k_{x_0}\rightarrow G/P$ at $x_0$. We will use the same symbol ${\bf p}_k$ for the evaluation map of $\mathcal T^k_{\mathfrak f(x_0)}$. Using Proposition~\ref{firstextension}, we will now prove that a germ of $Q$-cycle-respecting map at $x_0$ induces a biholomorphic map on $\mathcal T^k_{x_0}$ and a global section of the inverse image of $\mathscr B$ related in a similar way as in Proposition~\ref{firstextension}.

\begin{proposition}\label{kth extension}
Suppose $P$ is reduced mod $Q$ and vice versa. Let $x_0\in G/P$ and $\mathfrak f\in\mathscr B_{x_0}$ be a germ of $Q$-cycle-respecting map. For every $k\in\mathbb N^+$, there exist a biholomorphic map $F_k:\mathcal T^k_{x_0}\rightarrow\mathcal T^k_{\mathfrak f(x_0)}$ and a section $\mathfrak F_k\in ({\bf p}_k^{-1}\mathscr B)(\mathcal T^k_{x_0})$ such that under the canonical identification $({\bf p}_k^{-1}\mathscr B)_{\mu}\cong\mathscr B_{{\bf p}_k(\mu)}$ for $\mu\in\mathcal T^k_{x_0}$, we have (i) ${\bf p}_k(F_k(\mu))=(\mathfrak F_k)_\mu({\bf p}_k(\mu))$ for every $\mu\in\mathcal T^k_{x_0}$; and (ii) $(\mathfrak F_k)_\nu=\mathfrak f$ for every $\nu\in\mathcal D^k_{x_0}\subset\mathcal T^k_{x_0}$, where $\mathcal D^k_{x_0}$ is the diagonal preimage of $x_0$ in $\mathcal T^k_{x_0}$ (Definition~\ref{diagonal}).
\end{proposition}
\begin{proof}
By definition, $\mathcal T^1_{x_0}$ is the pullback bundle ${\bf q}_1^*\mathcal U$ of ${\bf q}:\mathcal U\rightarrow G/Q$ by ${\bf q}_1:{\bf p}^{-1}(x_0)\rightarrow G/Q$, where ${\bf q}_1$ is the restriction of $\bf q$ to ${\bf p}^{-1}(x_0)$. Since ${\bf q}_1:{\bf p}^{-1}(x_0)\rightarrow{\bf q}_1({\bf p}^{-1}(x_0))=\mathcal Q_{x_0}$ is a biholomorphism, we see that $\mathcal T^1_{x_0}$ is biholomorphic to ${\bf q}^{-1}(\mathcal Q_{x_0})$. Similarly, we also have $\mathcal T^1_{\mathfrak f(x_0)}\cong {\bf q}^{-1}(\mathcal Q_{\mathfrak f(x_0)})$.
On the other hand, following the notations and the context in Proposition~\ref{firstextension}, if $\zeta\in{\bf q}^{-1}(\mathcal Q_{x_0})$, then
$$
{\bf q}(\mathcal F(\zeta))={\bf q}([df^\sharp](\zeta))=f^\sharp({\bf q}(\zeta))\in f^\sharp(\mathcal Q_{x_0})\subset\mathcal Q_{\mathfrak f(x_0)}
$$
and thus $\mathcal F(\zeta)\in{\bf q}^{-1}(\mathcal Q_{\mathfrak f(x_0)})$. Therefore, the restriction of $\mathcal F$ gives a biholomorphism $F_1':{\bf q}^{-1}(\mathcal Q_{x_0})\rightarrow{\bf q}^{-1}(\mathcal Q_{\mathfrak f(x_0)})$, where $F_1':=\left.\mathcal F\right|_{{\bf q}^{-1}(\mathcal Q_{x_0})}$ and this canonically induces a biholomorphism $F_1:\mathcal T^1_{x_0}\rightarrow\mathcal T^1_{\mathfrak f(x_0)}$ through the identification $\mathcal T^1_{x_0}\cong{\bf q}^{-1}(\mathcal Q_{x_0})$.
Furthermore, under the same identification, the evaluation map ${\bf p}_1:\mathcal T^1_{x_0}\rightarrow G/P$ corresponds to the canonical projection ${\bf p}:{\bf q}^{-1}(\mathcal Q_{x_0})\rightarrow G/P$. Therefore, by again restricting to ${\bf q}^{-1}(\mathcal Q_{x_0})\cong\mathcal T^1_{x_0}$, the section $\mathfrak F$ obtained in Proposition~\ref{firstextension} gives a section $\mathfrak F_1\in ({\bf p}_1^{-1}\mathscr B)(\mathcal T^1_{x_0})$. The case for $k=1$ is thus settled.

Suppose now the biholomorphism $F_k$ and global section $\mathfrak F_k$ have been constructed for some $k\geq 1$. Consider the holomorphic fiber bundle $\pi:\mathcal T^{k+1}_{x_0}\rightarrow\mathcal T^k_{x_0}$, where $\pi:=\pi_{k+1,k}$, which is defined just before Definition~\ref{diagonal}. Similarly, write $\Pi:\mathcal T^{k+1}_{\mathfrak f(x_0)}\rightarrow\mathcal T^k_{\mathfrak f(x_0)}$ for the corresponding bundle at $\mathfrak f(x_0)$.

Let $\alpha\in\mathcal T^k_{x_0}$. As we have the canonical biholomorphisms $\pi^{-1}(\alpha)\cong\mathcal T^1_{{\bf p}_k(\alpha)}$ and $\Pi^{-1}(F_k(\alpha))\cong\mathcal T^1_{{\bf p}_k(F_k(\alpha))}$ and also
$(\mathfrak F_k)_\alpha({\bf p}_k(\alpha))={\bf p}_k(F_k(\alpha))$,
by using the germ $(\mathfrak F_k)_{\alpha}\in({\bf p}_k^{-1}\mathscr B)_\alpha\cong\mathscr B_{{\bf p}_k(\alpha)}$, we obtain, from the case $k=1$, a biholomorphism
$F^\alpha:\pi^{-1}(\alpha)\rightarrow \Pi^{-1}(F_k(\alpha))$,
and a section $\mathfrak F^\alpha\in ({\bf p}_{k+1}^{-1}\mathscr B)(\pi^{-1}(\alpha))$.
Since $\alpha\in\mathcal T^k_{x_0}$ is arbitrary, we have thus set-theoretically constructed a bijective map $F_{k+1}:\mathcal T^{k+1}_{x_0}\rightarrow\mathcal T^{k+1}_{\mathfrak f(x_0)}$ and a map $\mathfrak F_{k+1}$ from $\mathcal T^{k+1}_{x_0}$ to the espace \'etal\'e of ${\bf p}_{k+1}^{-1}\mathscr B$. We will now argue that $F_{k+1}$ is holomorphic and $\mathfrak F_{k+1}$ is a section in $({\bf p}_{k+1}^{-1}\mathscr B)(\mathcal T^{k+1}_{x_0})$.

Let $\mu\in\mathcal T^{k+1}_{x_0}$, $\alpha:=\pi(\mu)\in\mathcal T^k_{x_0}$ and $x:={\bf p}_k(\alpha)\in G/P$. Since $\mathfrak F_k\in ({\bf p}_k^{-1}\mathscr B)(\mathcal T^k_{x_0})$, there are an open neighborhood $\mathscr V\subset\mathcal T^k_{x_0}$ of $\alpha$, an open neighborhood $V\subset G/P$ of $x$, and a section $\mathfrak F_V\in\mathscr B(V)$ such that ${\bf p}_k(\mathscr V)\subset V$ and $(\mathfrak F_k)_\eta=(\mathfrak F_V)_{{\bf p}_k(\eta)}$ for every $\eta\in\mathscr V$.

By Proposition~\ref{firstextension}, we can, by choosing $V$ satisfying the condition ($\dagger$) therein,
obtain a biholomorphic map $\mathcal F:\mathcal V\rightarrow\mathcal F(\mathcal V)\subset\mathcal U$, where $\mathcal V:={\bf q}^{-1}({\bf q}({\bf p}^{-1}(V)))$, and a section $\mathfrak F\in({\bf p}^{-1}\mathscr B)(\mathcal V)$ satisfying the conditions stated in Proposition~\ref{firstextension}. Since $\pi:\mathcal T^{k+1}_{x_0}\rightarrow\mathcal T^k_{x_0}$ is continuous, there is a neighborhood ${\bf V}\subset\mathcal T^{k+1}_{x_0}$ of $\mu$ such that $\pi({\bf V})\subset\mathscr V$. Now for $\eta\in\bf V$, if we let $\beta:=\pi(\eta)\in\mathscr V$, since $\pi^{-1}(\beta)\cong\mathcal T^1_{{\bf p}_k(\beta)}\cong{\bf q}^{-1}(\mathcal Q_{{\bf p}_k(\beta)})$ canonically, $\eta$ corresponds to a point $\eta'\in{\bf q}^{-1}(\mathcal Q_{{\bf p}_k(\beta)})\subset\mathcal V$. We let $\varphi:{\bf V}\rightarrow\mathcal V$ be the  map defined by $\varphi(\eta)=\eta'$. From how we defined the $Q$-towers in Section~\ref{towersection}, $\varphi$ is actually $\widetilde{\bf q}_{k+1}|_{\bf V}$, and in particular, it is continuous.
By tracing back how we constructed $(\mathfrak F_{k+1})_\eta$, we see that $(\mathfrak F_{k+1})_\eta=\mathfrak F_{\varphi(\eta)}$. Since $\mathfrak F\in({\bf p}^{-1}\mathscr B)(\mathcal V)$, there are an open neighborhood $\mathcal W\subset\mathcal V$ of $\varphi(\mu)$, an open neighborhood $W\subset G/P$ of ${\bf p}(\varphi(\mu))$ and a section $\mathfrak F_W\in\mathscr B(W)$ such that ${\bf p}(\mathcal W)\subset W$ and $\mathfrak F_\lambda=(\mathfrak F_W)_{{\bf p}(\lambda)}$ for every $\lambda\in\mathcal W$. Now note that ${\bf p}\circ\varphi={\bf p}_{k+1}:{\bf V}\subset\mathcal T^{k+1}_{x_0}\rightarrow G/P$. So if we define the open set ${\bf W}:=\varphi^{-1}(\mathcal W)\subset{\bf V}\subset\mathcal T^{k+1}_{x_0}$, then ${\bf p}_{k+1}({\bf W})\subset W$ and for every $\eta\in{\bf W}$, we have
$$
(\mathfrak F_{k+1})_\eta=\mathfrak F_{\varphi(\eta)}=(\mathfrak F_W)_{{\bf p}(\varphi(\eta))}
=(\mathfrak F_W)_{{\bf p}_{k+1}(\eta)}.
$$
Since $\mu\in\mathcal T^{k+1}_{x_0}$ is arbitrary, we now see that $\mathfrak F_{k+1}$ is a section in $({\bf p}_{k+1}\mathscr B)(\mathcal T^{k+1}_{x_0})$.
Similarly, from $F_{k+1}(\eta) =\mathcal F (\varphi(\eta))$, it follows that $F_{k+1}$ is holomorphic.

Finally, for $\mu\in\mathcal T^{k+1}_{x_0}$, let $\alpha=\pi(\mu)$, then
$$
{\bf p}_{k+1}(F_{k+1}(\mu))={\bf p}_{k+1}(F^\alpha(\mu))=(\mathfrak F^\alpha)_\mu({\bf p}_{k+1}(\mu))=(\mathfrak F_{k+1})_\mu({\bf p}_{k+1}(\mu)).
$$
Moreover, if $\nu\in\mathcal D^{k+1}_{x_0}$, then $\beta:=\pi(\nu)\in\mathcal D^k_{x_0}$ and
$$
(\mathfrak F_{k+1})_\nu=(\mathfrak F^\beta)_\nu=(\mathfrak F_k)_\beta=\mathfrak f.
$$
\end{proof}

Using Proposition~\ref{relating analytic continuation}, we can interpret the germs $(\mathfrak F_k)_\mu$ as analytic continuations of $\mathfrak f$ along continuous curves in ${\bf p}_k(\mathcal T^k_{x_0})\subset G/P$, as follows.

\begin{corollary}\label{k extension along curves}
Let $k\in\mathbb N^+$, $\nu\in\mathcal D^k_{x_0}\subset\mathcal T^k_{x_0}$, $\mu\in\mathcal T^k_{x_0}$ and $\Gamma:[0,1]\rightarrow\mathcal  T^k_{x_0}$ be a continuous curve such that $\Gamma(0)=\nu$ and $\Gamma(1)=\mu$. The germ $(\mathfrak F_k)_\mu\in ({\bf p}_k^{-1}\mathscr B)_\mu\cong\mathscr B_{{\bf p}_k(\mu)}$ is the analytic continuation of $\mathfrak f$ along the curve ${\bf p}_k\circ\Gamma$.
\end{corollary}
\begin{proof}
Regard the global section $\mathfrak F_k$ as a continuous map from $\mathcal T^k_{x_0}$ to the espace \'etal\'e of ${\bf p}_k^{-1}\mathscr B$, then $\mathfrak F_k\circ\Gamma$ is a lifting of $\Gamma$  such that $\mathfrak F_k\circ\Gamma(0)=\mathfrak F_k(\nu)=\mathfrak f$. Thus, $\mathfrak F_k(\Gamma(1))=\mathfrak F_k(\mu)=(\mathfrak F_k)_\mu$ is the analytic continuation of $\mathfrak f\in({\bf p}_k^{-1}\mathscr B)_{\nu}$ along $\Gamma$. By Proposition~\ref{relating analytic continuation}, $(\mathfrak F_k)_\mu$ is the analytic continuation of $\mathfrak f$ along ${\bf p}_k\circ\Gamma$.
\end{proof}

\noindent\textbf{Remark.} For the case $k=1$ in Corollary~\ref{k extension along curves}, if we let $\mu=(x,s)\in\mathcal T^1_{x_0}\subset G/P\times G/Q$ (see the end of Section~\ref{towersection}) and choose $\nu=(x_0,s)\in\mathcal D^1_{x_0}={\bf p}_1^{-1}(x_0)$, then we can find a continuous curve $\gamma$ in $\mathcal P_s$ connecting $x_0$ and $x$. Such $\gamma$ can always be lifted to a curve $\Gamma$ in ${\bf q}^{-1}(s)\subset{\bf q}^{-1}(\mathcal Q_{x_0})\cong\mathcal T^1_{x_0}$ connecting $\nu$ and $\mu$. Thus, for every $\mu\in\mathcal T^1_{x_0}$, the germ $(\mathfrak F_1)_\mu$ is the analytic continuation of $\mathfrak f$ along a curve contained in a $Q$-cycle passing through $x_0$.

Now we can give the proof for Theorem~\ref{mainthm1}.

\begin{proof}[Proof of Theorem~\ref{mainthm1}]
We can always take $G$ to be a semisimple complex Lie group which is a direct product of simple factors. By Proposition \ref{parameter space}, there exists a unique parabolic subgroup $Q'$ properly containing $Q$ such that $Q'$ is reduced mod $P$ and we have the following diagram for the associated fibrations:
{\small
\begin{eqnarray*}
\xymatrix{
  &G/(P\cap Q)\ar[dl] \ar[d]  \\
     G/Q\ar[d]^{\omega}& G/(P\cap Q')\ar[dl] \ar[dr]\\
		G/Q'&&G/P
}
\end{eqnarray*}
}

For such $Q'$, every $Q'$-cycle in $G/P$ is a $Q$-cycle and vice versa. In addition, $G/Q'$ effectively parameterizes the $Q'$-cycles (or $Q$-cycles). More explicitly, for $s, t\in G/Q$, we have $\mathcal P_s=\mathcal P_t$ if and only if $\omega(s)=\omega(t)$.
Since the $Q$-cycles and $Q'$-cycles on $G/P$ coincide, it suffices to prove the proposition for $P$ and $Q'$. Thus, without loss of generality, we can assume that $Q$ is reduced mod $P$.

Let $x_0\in G/P$ and $\mathfrak f\in\mathscr B_{x_0}$ be a germ of $Q$-cycle-respecting map. We first prove the theorem for the cases where $P$ is reduced mod $Q$.

By Proposition~\ref{connectivity}, there exists $N\in\mathbb N^+$ such that the evaluation map  ${\bf p}_N:\mathcal T^N_{x_0}\rightarrow G/P$ from the $N$-th $Q$-tower is surjective. Let $F_N:\mathcal T^N_{x_0}\rightarrow\mathcal T^N_{\mathfrak f(x_0)}$ be the biholomorphism and $\mathfrak F_N\in({\bf p}_N^{-1}\mathscr B)(\mathcal T^N_{x_0})$ be the global section obtained from $\mathfrak f$ in Proposition~\ref{kth extension}.

By Proposition~\ref{fiber connectedness}, for any $x\in G/P$, the preimage ${\bf p}_N^{-1}(x)$ is connected. Thus, it follows from Corollary~\ref{constant on fibers} that $(\mathfrak F_N)_\mu=(\mathfrak F_N)_\nu$ for every $\mu,\nu\in{\bf p}_N^{-1}(x)$. Using (i) of Proposition~\ref{kth extension}, we now see that
$$
{\bf p}_N(F_N(\mu))=(\mathfrak F_N)_\mu({\bf p}_N(\mu))
=(\mathfrak F_N)_\nu({\bf p}_N(\nu))={\bf p}_N(F_N(\nu))
$$
for every $\mu,\nu\in{\bf p}_N^{-1}(x)$. That is, the map ${\bf p}_N\circ F_N:\mathcal T^N_{x_0}\rightarrow G/P$ is constant on each preimage ${\bf p}_N^{-1}(x)$. Hence, $F_N$ descends to a holomorphic map $F:G/P\rightarrow G/P$, which is actually biholomorphic because the same arguments apply to $F_N^{-1}$, which clearly descends to the inverse of $F$.

We are now going to prove that $F$ indeed extends the germ $\mathfrak f$. First of all, note that by the definition of the inverse image sheaf ${\bf p}_N^{-1}\mathscr B$, the statement $(ii)$ in Proposition~\ref{kth extension} implies for each $\nu\in\mathcal D^N_{x_0}$, there are an open neighborhood $\mathcal V\subset\mathcal T^N_{x_0}$ of $\nu$ , an open neighborhood $V\subset G/P$ of ${\bf p}_N(\mathcal V)\ni x_0$ and a $Q$-cycle-respecting map $f:V\rightarrow G/P$ representing $\mathfrak f$, such that $(\mathfrak F_N)_\mu({\bf p}_N(\mu))=f({\bf p}_N(\mu))$ for every $\mu\in\mathcal V$. Together with statement $(i)$ in Proposition~\ref{kth extension}, we conclude that for every $x$ in the open set ${\bf p}_N(\mathcal V)\subset G/P$ and any $\mu\in{\bf p}_N^{-1}(x)$,
$$
F(x)={\bf p}_N(F_N(\mu))=(\mathfrak F_N)_\mu({\bf p}_N(\mu))=f({\bf p}_N(\mu))=f(x).
$$
Thus, $F$ extends $\mathfrak f$ and the proof is complete for the cases where $P$ is reduced mod $Q$.



It remains to consider the general case where $P$ is not reduced mod $Q$. Define an equivalent relation $\approx$ on $G/P$ by $x \approx y$ for $x,y \in G/P$ if and only if $\mathcal Q_x =\mathcal Q_y$, or equivalently, $\cap_{x \in \mathcal P_s} \mathcal P_s = \cap_{y \in \mathcal P_s} \mathcal P_s$. (One direction is clear: if $\mathcal Q_x = \mathcal Q_y$, then $\cap_{x \in \mathcal P_s} \mathcal P_s = \cap_{y \in \mathcal P_s} \mathcal P_s$. To prove the converse, suppose that  $\cap_{x \in \mathcal P_s} \mathcal P_s = \cap_{y \in \mathcal P_s} \mathcal P_s$. Since $y \in \cap_{y \in \mathcal P_s} \mathcal P_s$, $y$ is contained in $\cap_{x \in \mathcal P_s} \mathcal P_s $, i.e., $y$ is contained in $\mathcal P_s$ for all $s \in \mathcal Q_x$. Hence $\mathcal Q_x \subset \mathcal Q_y$ and thus $\mathcal Q_x =\mathcal Q_y$.) Then the relation $\approx$ is equivariant under the action of $G$, that is, for any $g \in G$, $x \approx y$ if and only if $gx \approx gy$. Hence there is a $G$-equivariant quotient map $\varpi$ from $G/P$ to the space $\tilde X $ of equivalent classes and, if we let $\tilde P $ be the isotropy group of $G$ at the equivalent class containing the base point of $G/P$ at which the isotropy group is $P$, then $P$ is contained in $\tilde P $ and we have $\tilde X  = G/\tilde P $. (In fact, $\tilde P$ is nothing but the reduction of $P$ mod $Q$.)
%
%
 For $s \in G/Q$ let $\tilde{\mathcal P}_s:=\varpi(\mathcal P_s)$. Then $\tilde{\mathcal P}_s$ are the $Q$-cycles on $G/\tilde P$. Now $\tilde x \in G/\tilde P$ is determined by $\mathcal Q_{\tilde x}$. 

Let $f:V_1 \rightarrow V_2$ be a local biholomorphism satisfying the condition ($\dag$) in the proof of Proposition~\ref{firstextension}. Then $f$ induces a biholomorphism  $\tilde f :\tilde V_1:=\varpi(V_1) \rightarrow \tilde V_2:=\varpi(V_2)$, and $\tilde f$ satisfies the condition ($\dag$) with respect to the family $\{\tilde{\mathcal P}_s: s \in G/Q\}$. Note that we still have $Q$-cycle-connectivity for the $Q$-cycles on $G/\tilde P$ since such property is inherited by $G$-equivariant projection. Now $P$ is also reduced mod $Q$, so we conclude that $\tilde f$ extends to a biholomorphism $\tilde F: G/\tilde P \rightarrow G/\tilde P$ by the case we settled before.

The automorphism $\tilde F$ is given by some element of $G$ except in the following cases:
\begin{enumerate}
\item $G/\tilde P=\mathbb {CP}^{2n-1}$ and $G=Sp(n, \mathbb C)$;
\item $G/\tilde P$ is a spinor variety and $G=SO(2n+1, \mathbb C)$;
\item $G/\tilde P=\mathbb Q^5$ and $G=G_2$.
\end{enumerate}

If $G/\tilde P$ is not one of (1) - (3), then $\tilde F$ comes from an element $g\in G$ and thus $f$ is also a restriction of the automorphism of $G/P$ given by the same $g$.

If $G/\tilde P$ is   one of (1) - (3), then the marking $\phi_{\tilde P}$ associated to $\tilde P$ on the Dynkin diagram of $G$ consists of only one element which is an end of the diagram.
Let $\phi_Q$ be the corresponding marking for $Q$. Since  $ \phi_{\tilde P} $ is the reduction of $\phi_P$ mod $\phi_Q$, it is contained in $\phi_P$. If $\phi_{\tilde P}\subsetneq\phi_P$, then there is a connected subdiagram of $G$ containing both a node in $\phi_P - \phi_{\tilde P}$ and a node in $\phi_Q$ which does not contain any element in $\phi_{\tilde P}$, contradicting to the fact that $\phi_{\tilde P}$ is the reduction of $\phi_P$ mod $\phi_Q$. Therefore, $\phi_{\tilde P} =\phi_P$. Hence, $G/\tilde P=G/P$ and $\tilde F$ is an extension of $f$.
%
\end{proof}

\section{Local biholomorphisms preserving real group orbits}


\subsection{Real group orbits of flag type}\label{flag type section}$\,$

 Let $G/P$ be a rational homogeneous space and let $G_0$ be a  real form of $G$, i.e. $G_0$ is the real analytic subgroup of $G$ corresponding to a real Lie subalgebra $\mathfrak g_0\subset\mathfrak g$ such that $\mathfrak g_0\oplus J\mathfrak g_0=\mathfrak g$, where $\mathfrak g$ is the Lie algebra of $G$ and $J$ is the complex structure operator. Let $\tau$ denote the complex conjugation of   $G$ over $G_0$.

We now collect some foundational results about the canonical action of $G_0$ on $G/P$ established by Wolf~\cite{W69}. First of all, there are only a finite number of orbits, which will be called \textit{real group orbits}. In particular, there are open orbit(s). Furthermore, it is also known that there is a unique closed orbit.

Let $\mathcal O$ be a $G_0$-orbit in $G/P$. A \textit{holomorphic arc} in $\mathcal O$ is a holomorphic map $f:\Delta\rightarrow G/P$ such that $f(\Delta)\subset\mathcal O$, where $\Delta$ is the unit disk in $\mathbb C$. By a \textit{chain of holomorphic arcs in} $\mathcal O$ \textit{connecting} two points $x, y$ in $\mathcal O$,  we mean a finite sequence of holomorphic arcs $f_1, \ldots, f_k$ in $\mathcal O$ such that $x \in  f_1(\Delta)$ and $y \in   f_k(\Delta)$ and $  f_i(\Delta) \cap  f_{i+1}(\Delta) \not=\emptyset$ for all $i=1, \ldots, k-1$. Define an equivalence relation on $\mathcal O$ as follows. Two elements $x,y\in\mathcal O$ are equivalent if and only if there is a chain of holomorphic arcs in $\mathcal O$ connecting $x$ and $y$.  An equivalence class of this equivalence relation on $\mathcal O$ is called a {\it holomorphic arc component} of $\mathcal O$. We remark that a holomorphic arc component is not necessarily a complex submanifold of $\mathcal O$.

Fix $x\in\mathcal O$ and let $C$ be the holomorphic arc component of $\mathcal O$ containing $x$. Denote by $N_0$ the identity component of the normalizer $N_{G_0}(C):=\{g \in G_0: gC  = C\}$ of $C $ in $G_0$. Then $C$ is an $N_0$-orbit (Lemma 8.2  of \cite{W69}). The complexification $N$ of $N_0 $ is a parabolic subgroup of $G$ such that $\tau N  = N $. Moreover, $N_{0}$ is the identity component of the parabolic subgroup $N \cap G_0$ of $G_0$ (Theorem 8.5 of \cite{W69}). Let $s\in G/N $ be the point whose isotropy group is $N$ and let $\Sigma$ be  the $G_0$-orbit containing $s$ in $G/N$. Since $\tau N  = N $  ,   $\Sigma$ is closed in $G/N$ and is totally real in the sense that $\Sigma$ is the set of real points of the complex projective variety $G/N$ defined over $\mathbb R$~(Theorem 3.6 of \cite{W69}).

\begin{definition}[\cite{W69}, Definition 9.1]\label{flag type}
A real group orbit $\mathcal O$ is said to be \textit{partially complex} if its holomorphic arc components are locally closed complex submanifolds of $G/P$; and of \textit{flag type} if for $x\in\mathcal O$, the orbit $Nx$ is a rational homogeneous space, i.e. if $N\cap P_x$ is parabolic, where $P_x\subset G$ is the isotropy group of $x$.
\end{definition}

\noindent\textbf{Remark.} The criterion for $\mathcal O$ being of flag type given in~\cite{W69} is slightly different from the one given here but they are equivalent.

As mentioned, if $C$ is the holomorphic arc component containing $x$, then $C=N_0x$. Therefore, $C$ is a real group orbit of the rational homogeneous space  $Nx$ if $\mathcal O$ is of flag type.

\begin{example}[see Remark 9.23 in \cite{W69}] The following are examples of $G_0$-orbits which are partially complex and of flag type. \label{example of flag type}
\begin{enumerate}
 \item The $G_0$-orbits in a compact Hermitian symmetric space $G/P$, where $G_0$ is the automorphism group of the bounded symmetric domain dual to $G/P$. It is well known that every holomorphic arc component is isomorphic to some bounded symmetric domain except for the closed orbit. \\

\item The orbits of any $SU(p,q)$ acting $Gr(n, \mathbb C^{2n})$, where $p+q=2n$ and $p<q$. For example, for the action of $SU(1,3)$ on $Gr(2,\mathbb C^4)$, the closed orbit has holomorphic arc components isomorphic to $\mathbb P^1$.
\end{enumerate}
\end{example}

\subsection{Holomorphic cover of subdiagram type}$\,$

Now we are going to apply our extension theorem for $Q$-cycle-respecting maps to study local holomorphic maps preserving a real group orbit on a rational homogeneous space $G/P$. Let $x_0\in G/P$ be the point whose isotropy group is $P$. Let $G_0\subset G$ be a real form and $\mathcal O$ be the $G_0$-orbit containing $x_0$. Suppose $Q \subset G$ is a parabolic subgroup such that $P \cap Q$ is parabolic. As before we have the double fibration $G/Q\overset{\bf q}\longleftarrow G/(P\cap Q)\overset{\bf p}\longrightarrow G/P$. Now let $\mathcal S$ be a $G_0$-orbit on $G/Q$ such that $x_0\in{\bf p}({\bf q}^{-1}(\mathcal S))$. That is, we have
\begin{eqnarray*}
\xymatrix{
  & G/(P\cap Q )\ar[dl]_{\mathbf q} \ar[dr]^{\mathbf p}  \\
     \mathcal S \subset G/Q &&G/P  \supset \mathcal O\ni x_0.
 }
\end{eqnarray*}


We begin with the following definition, which roughly describes the situation in which a real group orbit can be written as  a union of complex submanifolds lying on a family $Q$-cycles which are rigid under holomorphic mappings.

\begin{definition} \label{holomorphic cover}
We say that $\mathcal O$ has a \textit{holomorphic cover of $Q$-type}  if there exists a $G_0$-orbit $\mathcal S\subset G/Q$ such that,
for every $s\in\mathcal S$, there is a real group orbit $\mathcal O_s$ on $\mathcal P_s:={\bf p}({\bf q}^{-1}(s))$  satisfying

$(i)$ $\mathcal O=\bigcup_{s\in\mathcal S}\mathcal O_s$; and

$(ii)$ for each $s\in\mathcal S$ and any holomorphic map $h:W\rightarrow G/P$ defined on an open set $W\subset\mathcal P_s$ with $h(W\cap\mathcal O_s)\subset\mathcal O$, we have that each connected component of $h(W\cap\mathcal O_s)$ is contained in $\mathcal O_t$ for some $t\in\mathcal S$.


\end{definition}

\begin{definition}
We say that $\mathcal O$ has a {\it holomorphic cover of subdiagram type} if it has a holomorphic cover of $Q$-type for some parabolic subgroup $Q\subset G$; and such a cover is said to be \textit{non-trivial} if in addition $Q\not\subset P$ and $Q\neq G$.
\end{definition}

The following proposition is the major motivation for the definitions above.

\begin{proposition}
If a real group orbit $\mathcal O\subset G/P$ is partially complex and of flag type, then it has a holomorphic cover of subdiagram type.
\end{proposition}
\begin{proof}
Recall the notations in Section~\ref{flag type section}. Let $C\subset\mathcal O$ be a holomorphic arc component containing a point $x\in\mathcal O$ and $Q$ be the complexification of the normalizer of $C$ in $G_0$. We have known that $C\subset Qx$ is a real group orbit of $Qx$ which is a rational homogeneous space because $\mathcal O$ is of flag type. Furthermore, since $\mathcal O$ is partially complex, $C$ is a complex submanifold of $G/P$. Let $s\in G/Q$ be a point such that the corresponding $Q$-cycle is $Qx$. Define $\mathcal S:=G_0s\subset G/Q$ and for every $s'\in\mathcal S$, let $\mathcal P_{s'}:={\bf p}({\bf q}^{-1}(s'))$. Then if we write $s'=gs$, where $g\in G_0$, then $C_{s'}:=gC$ is independent of the choice of $g$ and also we have $C_{s'}\subset \mathcal P_{s'}\cap\mathcal O$, which is the holomorphic arc component containing $gx$.

It remains to verify that the family $\{\mathcal P_{s'}:s'\in\mathcal S\}$ of $Q$-cycles together with  the family $\{\mathcal O_{s'}:=C_{s'}\subset\mathcal P_{s'}:s'\in\mathcal S\}$ of real group orbits (which are complex submanifolds), is a holomorphic cover of $Q$-type for $\mathcal O$. Let $s'\in\mathcal S$ and $W\subset\mathcal P_{s'}$ be a connected open set, and $h:W\rightarrow G/P$ be a holomorphic map such that $h(W\cap\mathcal O_{s'})\subset\mathcal O$. Since $\mathcal O_{s'}$ is a complex submanifold, it follows that each connected component of $h(W\cap\mathcal O_{s'})$ is contained in some holomorphic arc component of $\mathcal O$ and thus in $\mathcal O_t$ for some $t\in\mathcal S$.
\end{proof}

\begin{example} \label{example} The following  real group orbits have a holomorphic cover of subdiagram type.
\begin{enumerate}
\item By the previous proposition, any real group orbit which is partially complex and of flag type has a holomorphic cover of $Q$-type, where $Q$ is the complexification of the normalizer of a holomorphic arc component of it.
Among the $G_0$-orbits in Example \ref{example of flag type}  (1), an orbit has a non-trivial holomorphic cover of subdiagram type if it is neither open nor closed. For the $G_0$-orbits in Example \ref{example of flag type}  (2), the closed orbit has a non-trivial holomorphic cover of subdiagram type.\\

\item  The closed orbit of $G_0=SU(p,q)$ on $Gr(d, \mathbb C^{p+q})$, where $d<\min(p,q)$, has a non-trivial holomorphic cover of subdiagram type (\cite{NG12}, Proposition 3.2 therein). In this case,  the holomorphic cover is of $Q$-type, where $Q$ is the maximal parabolic subgroup of $G=SL(p+q, \mathbb C)$ such that $G/Q=Gr(\min(p,q),\mathbb C^{p+q})$. Moreover, in this case we have $\mathcal O_s=\mathcal P_s\cong Gr(d,\min(p,q))$.

  \end{enumerate}
\end{example}

\subsection{Proof of Theorem~\ref{theorem extension of subdiagram type}}$\,$

We begin with the following simple observation that a real group orbit is always a set of uniqueness for holomorphic functions.

\begin{lemma}\label{subvariety lemma}
Let $\mathcal O$ be a real group orbit on $G/P$ for some real form $G_0\subset G$. Let $U\subset G/P$ be a connected open subset such that $U\cap\mathcal O\neq\emptyset$. Then there does not exist any proper complex analytic subvariety of $U$ containing $U\cap\mathcal O$.
\end{lemma}
\begin{proof}
Let $x\in \mathcal O$, $T_x(\mathcal O)$ be the real tangent space of $\mathcal O$ at $x$ and $J$ be the complex structure operator.  Note that we have the canonical isomorphisms $T_x(\mathcal O)\cong \mathfrak g_0/(\mathfrak g_0\cap\mathfrak p_x)\cong(\mathfrak g_0+\mathfrak p_x)/\mathfrak p_x$, where $\mathfrak p_x\subset\mathfrak g$ is the Lie algebra of the isotropy group $P_x$ at $x$, then
\begin{eqnarray*}
T_x(\mathcal O)+ J(T_x(\mathcal O))&\cong& (\mathfrak g_0+\mathfrak p_x)/\mathfrak p_x+J(\mathfrak g_0+\mathfrak p_x)/\mathfrak p_x\\
&=&(\mathfrak g_0+J\mathfrak g_o+\mathfrak p_x)/\mathfrak p_x\\
&=&\mathfrak g/\mathfrak p_x\cong T_x(G/P)
\end{eqnarray*}
Hence, we see that $T_x(\mathcal O)$ cannot be contained in a $J$-invariant proper tangent subspace at $x$ and the desired statement follows.
\end{proof}

 \begin{proposition} \label{holomorphic function vanishing on real points}
Let $\mathcal O$ be a real group orbit on $G/P$. Suppose $U\subset G/P$ is a connected open subset such that $U \cap \mathcal O \not=\emptyset$ and $u: U \rightarrow \mathbb C$ is a holomorphic function such that $u|_{U \cap\mathcal O} \equiv 0$, then $u$ is identically zero on $U$.
 \end{proposition}

\begin{proof}
It follows directly from Lemma~\ref{subvariety lemma} since the zero set of any non-trivial holomorphic function on $U$ is a proper complex analytic subvariety in $U$.
\end{proof}

\begin{proof}[Proof of Theorem \ref{theorem extension of subdiagram type}]
Let $\mathcal O$ be a $G_0$-orbit in $G/P$   having a non-trivial holomorphic cover of subdiagram type. Then  there is a parabolic subgroup $Q \subsetneq G$ such that $P \cap Q$ is parabolic and we have the double fibration
\begin{eqnarray*}
\xymatrix{
  &G/(P \cap Q )\ar[dl]_{\mathbf q} \ar[dr]^{\mathbf p}  \\
     \mathcal S \subset G/Q &&G/P  \supset \mathcal O
 }
\end{eqnarray*}
where $\mathcal S \subset G/Q$ is a real group orbit,
and a family $\{\mathcal O_s\}_{s \in \mathcal S}$ of real group orbits of  $\mathcal P_s:={\bf p}({\bf q}^{-1}(s))$ such that
 \begin{enumerate}
 \item  $\mathcal O=\bigcup_{s\in\mathcal S}\mathcal O_s$;

 \item  for each $s\in\mathcal S$ and any holomorphic map $h:W\rightarrow G/P$ defined on an open set $W\subset\mathcal P_s$, satisfying $h(W\cap\mathcal O_s)\subset\mathcal O$, we have that each connected component of $h(W\cap\mathcal O_s)$ is contained in $\mathcal O_t$ for some $t\in\mathcal S$.
\end{enumerate}
 Replacing $U$ by a smaller open subset if necessary, we may assume that $\mathcal P_s \cap U$ is connected for any $\mathcal P_s$ intersecting $U$ (see the proof of Proposition \ref{good neighborhood}).

Since  $f$ is a biholomorphism such that $f(U \cap \mathcal O) \subset \mathcal O$, if we let $s\in\mathcal S$ and fix a connected component of $ \mathcal O_s   \cap U$, its image  under $f$   is contained in some $\mathcal O_t\subset\mathcal P_{t}$   by condition (2).
Now, since $\mathcal O_s$ is a real group orbit of $\mathcal P_s$, following directly from Proposition \ref{holomorphic function vanishing on real points}, we deduce that the image of    a connected component of $ \mathcal P_s \cap U$ under $f$ is also contained in some  $\mathcal P_{t}$.
We are going to show that  this is true for every $Q$-cycle intersecting $U$, i.e., $f:U \subset G/P \rightarrow G/P$ is in fact $Q$-cycle-respecting.

As in the proof of Theorem \ref{mainthm1}, identify the universal family $\mathcal U=G/(P \cap Q)$ with a closed complex submanifold of the Grassmannian bundle $ Gr(k, T(G/P)) $, where $k=\dim_{\mathbb C}(\mathcal P_s)$. Let $V:=f(U)$. Then the differential $df$ induces a map $[df]:Gr(k, T(U)) \rightarrow Gr(k,T(V))$. We claim that

\begin{enumerate}
\item [{\rm (a)}] $[df]$ maps ${\mathbf p}^{-1}(U) $   to ${\mathbf p}^{-1}(V)$;
\item  [{\rm (b)}] $[df]:{\mathbf p}^{-1}(U) \rightarrow {\mathbf p}^{-1}(V)$ is fiber-preserving with respect to $\mathbf q$, i.e., it sends   a fiber  of $\mathbf q$ to a fiber of $\mathbf q$.
\end{enumerate}

Since ${\mathbf p}^{-1}(V) \subset Gr(k, T(V))$ is a closed complex submanifold, to say that $[df]({\mathbf p}^{-1}(U)) \subset {\mathbf p}^{-1}(V)$ is the same as saying that the pullbacks of the local defining holomorphic functions for ${\mathbf p}^{-1}(V)$ vanish on ${\mathbf p}^{-1}(U)$.

Let $(s_0,x_0) \in {\mathbf p}^{-1}(U)$, where $x_0\in G/P$ (respectively,  $s_0 \in G/Q$) is the base point of $G/P$ with the isotropy group $P$ (respectively, $G/Q$ with the isotropy group $Q$).
Take connected open neighborhoods $\mathscr S \subset G/Q$ of $s_0$ and $\chi \subset \mathcal O_{s_0}$ of $x_0$ so that $\mathscr S \times \chi$ can be embedded into ${\mathbf p}^{-1}(U)$  as an open neighborhood of $(s_0,x_0 )$ with ${\mathbf q}(\mathscr S \times \chi )=\mathscr S$. For each $x  \in \chi$ consider the intersection ${\mathbf q}^{-1}(\mathcal S) \cap (\mathscr S \times \{x \})=(\mathcal S \cap \mathscr S) \times \{x \}$. We observe that each point $(s,x ) \in (\mathcal S \cap \mathscr S) \times \{x \}$ as a point in the universal family, corresponds to a holomorphic tangent space of $\mathcal P_s$ at $x $ (since $\mathcal S$ parametrizes such family). Furthermore, as proven above, the image  $f(\mathcal P_s \cap U)$ is contained in some $\mathcal P_t  $. Hence we have
$$[df]((\mathcal S \cap \mathscr S) \times \{x \}) \subset {\mathbf p}^{-1}(V).$$

Now the pullbacks of the local defining functions of ${\bf p}^{-1}(V)$ vanish on $(\mathcal S \cap \mathscr S) \times \{x \}$ for every $x  \in \chi$.  Proposition \ref{holomorphic function vanishing on real points}  implies that these holomorphic functions vanish identically on $\mathscr S \times \{x \}$ for every $x  \in \chi$ and hence on an open neighborhood in ${\mathbf p}^{-1}(U)$. Therefore, $[df]({\mathbf p}^{-1}(U)) \subset {\mathbf p}^{-1}(V)$. 

To see (b)  we first consider the case where $s \in \mathcal S \cap  U^{\sharp}$, where $U^{\sharp}:={\mathbf q}({\mathbf p}^{-1}(U))$. By the arguments at the beginning of the proof,
  we know that the fibers of $\mathbf q$ over $\mathcal S \cap  U^{\sharp}$ are preserved by $[df]$.
  On the other hand, the fiber-preserving property can clearly be translated to the vanishing of a set of holomorphic functions defined on $U^{\sharp}$. We have just seen that these relevant holomorphic functions vanish on $\mathcal S \cap U^{\sharp}$ and Proposition~\ref{holomorphic function vanishing on real points} once again implies that they vanish identically on $ U^{\sharp}$ and therefore $[df]$ is everywhere fiber-preserving with respect to $\mathbf q$, which simply translates to the statement that $f$ is $Q$-cycle-respecting.

We have thus shown that $f$ is a $Q$-cycle-respecting local biholomorphism and by Theorem \ref{mainthm1}, $f$ extends to a biholomorphism of $G/P$.
\end{proof}

\section{Local biholomorphisms respecting tangent spaces of $Q$-cycles}\label{last section}

Recall that $G/(P \cap Q)$ can be regarded as a closed complex submanifold $\mathcal C$ of $Gr(k, T(G/P))$, namely, as the variety of tangent spaces of $Q$-cycles on $G/P$.  We may ask whether we can have an extension theorem parallel to the original Cartan-Fubini extension Theorem for minimal rational curves. That is, does a local biholomorphism that is only known to respect   tangent spaces of the $Q$-cycles extend to a global biholomorphism? In this section we will show that the answer is affirmative when $G/P$ is of Picard number one and give a proof of Theorem~\ref{cartan-fubini}.

The question is whether the preservation of the variety of tangent spaces of $Q$-cycles implies the $Q$-cycle-respecting property. It can be rephrased in the following way.

\noindent \textbf{Question.}   Let $S \subset G/P$ be a locally closed, connected complex submanifold which is an integral variety of $\mathcal C\subset Gr(k, T(G/P))$, i.e., for every $p\in S$, there is a $Q$-cycle, depending on $p$, sharing the same tangent subspace with $S$ at $p$. Is $S$ itself an open subset of a $Q$-cycle?

 An affirmative answer to the above question is a sufficient to prove Theorem~\ref{cartan-fubini}. However, in the case of minimal rational curves, i.e. when the $Q$-cycles are lines, then there exists some holomorphic map $h:\mathbb P^1 \rightarrow G/P$ such that $h(\mathbb P^1)$ is an integral variety of $\mathcal C \subset Gr(1, T(G/P))=\mathbb P(T(G/P))$ and $h(\mathbb P^1)$ is not a line (Section 6 of \cite{CH}). Nevertheless, the Cartan-Fubini extension for minimal rational curves still holds.

If $Q$-cycles are either non-linear or a maximal linear subspace,
the answer to above question is affirmative when $G/P$ is of Picard number 1 with some exceptions, described as follows.

\begin{proposition}  \label{integral varieties} Let $G/P$ be a rational homogeneous space of Picard number 1 and let $Q\subset G$ be a parabolic subgroup such that $P\cap Q$ is parabolic. Assume that the $Q$-cycles are either maximal linear subspaces of $G/P$ or non-linear homogeneous subspaces.
If $S\subset G/P$ is a locally closed, connected complex submanifold which is an integral variety of $G/(P \cap Q)\subset Gr(k, T(G/P))$, then $S$ is an open subset of a $Q$-cycle unless
 \begin{enumerate}
 \item $G/P$ is of type $(B_{\ell}, \alpha_i)$ and  $Q$    is associated with $\{\alpha_{i-1}, \alpha_{\ell}\}$;

 \item $G/P$ is of type $(C_{\ell}, \alpha_{\ell})$ and  $Q$  is associated with  $\{\alpha_{\ell-1} \}$;

 \item $G/P$ is of type $(F_{4},\alpha_1)$   and  $Q$  is associated with  $\{\alpha_{3} \}$;

 \item $G/P$ is of type $(G_{2},\alpha_2)$   and  $Q$  is associated with  $\{\alpha_{1} \}$.
 \end{enumerate}
\end{proposition}

\begin{proof}
The desired statement is included in Corollary 1.2 and Lemma 1.2 of \cite{MokZhang}.
 \end{proof}

 \noindent \textit{Proof of Theorem~\ref{cartan-fubini}}.
If the $Q$-cycles are either non-linear homogeneous subspaces or maximal linear subspaces of $G/P$ other than the exceptions listed in Proposition \ref{integral varieties}, any local biholomorphism $f:U \rightarrow f(U)\subset G/P$ preserving the variety of tangent spaces of $Q$-cycles is $Q$-cycle-respecting by Proposition \ref{integral varieties}. By Theorem \ref{mainthm1}, $f$ can be extended to a biholomorphism of $G/P$.

If the $Q$-cycles are non-maximal linear subspaces of $G/P$ or maximal linear spaces in the list of  Proposition \ref{integral varieties}, then the preservation of the variety of tangent spaces of $Q$-cycles implies the preservation of the variety of tangent spaces of {\it invariant} lines (see the remark below), and thus, by Corollary 5.4 of \cite{Y}, $f$ can be extended to a biholomorphism of $G/P$.
 \qed

\noindent \textbf{Remark.} If $P$ is associated to a long root, then lines in $G/P$ are $Q$-cycles for some $Q$. However, if $P$ is associated to a short root, then   generic lines   are not $Q$-cycles. By an {\it invariant} line we mean a line that is a $Q$-cycle.
The Cartan-Fubini extension Theorem of Hwang-Mok~\cite{HM01} can be applied to a local biholomorphism preserving the set of tangent directions of lines in  $G/P$, while Corollary 5.4 of \cite{Y} can be applied to a local biholomorphism preserving the set of tangent directions of invariant lines in $G/P$.

$\,$

\noindent\textbf{Acknowledgements.} We would like to thank Professor Jun-Muk Hwang for raising a question about the univalence of extension which led us to discover a gap in the original proof of Theorem~\ref{mainthm1}. We are very grateful to Professor Ngaiming Mok for a lot of precious advice regarding the construction of $Q$-towers. We also like to thank Professor Joeseph Wolf for helpful comments.

\end{document}